\documentclass[11pt]{article}

\usepackage{GoetzArticle}
\usepackage{amssymb}
\usepackage{verbatim}
\usepackage{amssymb}
\usepackage{amsmath}
\usepackage[pdftex]{graphicx}
\usepackage{epsfig}
\usepackage{epsf}
\usepackage[pdftex]{color}
\usepackage{rotating}
\usepackage{bm}

\usepackage{rawfonts}
\usepackage{bm}
\input prepictex
\input pictexwd
\input postpictex

\newcommand{\T}{\mathbb{T}}

\newcommand{\supp}{{\rm supp}\, }
\newcommand{\sgn}{{\rm sgn}}

\newcommand{\R}{\mathbb{R}}

\newcommand{\N}{\mathbb{N}}

\newcommand{\C}{\mathbb{C}}

\newcommand{\RR}{{\mathbb \R}}

\renewcommand{\P}{{\mathbb{P}}}
\newcommand{\E}{{\mathbb{E}}}
\newcommand{\spacing}[1]{\renewcommand{\baselinestretch}{#1}\large\normalsize}
\spacing{1.795}

\newcommand{\Id}{ \boldsymbol{Id} }

\newcommand{\nrecA}{\text{BP does not recover } \op \text{ from } \op \boldsymbol{h}^A}
\newcommand{\recR}{ \text{BP recovers } \op \text{ from } \op \boldsymbol{h}^R}
\newcommand{\nrecR}{\text{BP does not recover } \op \text{ from } \op \boldsymbol{h}^R}

\newcommand{\BPsi}{\bm{\Psi}}
\newcommand{\BG}{\bm{\mathcal G} }

\def\op{\boldsymbol \Gamma}
\def\CC{{\mathbb C}}

\newcommand{\defright}{=}
\newcommand{\defleft}{=}

\title{Identification of Matrices having a Sparse Representation}
\author{
        G\"otz E. Pfander \footnotemark[1],
        Holger Rauhut     \footnotemark[2],
        Jared Tanner   \footnotemark[3]
        }

\begin{document}
\maketitle
\date{}
\renewcommand{\thefootnote}{\fnsymbol{footnote}}

\footnotetext[1]{School of Engineering and Science, Jacobs
University Bremen, 28759 Bremen, Germany, g.pfander@jacobs-university.de}

\footnotetext[2]{Numerical Harmonic Analysis Group, Faculty of
Mathematics, University of Vienna, Nordbergstrasse 15, A-1090
Vienna, Austria, holger.rauhut@univie.ac.at. H.R. is supported by
the European Union's Human Potential Programme under contract
MEIF-CT 2006-022811.}

\footnotetext[3]{ Department of Mathematics, University of Utah, 155
South 1400 East, Salt Lake City, UT 84112-0090, USA.  J.T. would like
to thank John E. and Marva M. Warnock for their generous support in the
form of an endowed chair.
tanner@math.utah.edu}
\renewcommand{\thefootnote}{\arabic{footnote}}
\pagenumbering{arabic} \maketitle



\begin{center}
\begin{minipage}[c]{15cm} \it
We consider the problem of recovering a matrix from its action on a
known vector in the setting where the matrix can be represented
efficiently in a known matrix dictionary. Connections with sparse
signal recovery allows for the use of efficient reconstruction
techniques such as Basis Pursuit. Of particular interest is the
dictionary of time-frequency shift matrices and its role for
channel estimation and identification in communications engineering.  We present
recovery results for Basis Pursuit with the time-frequency shift dictionary
and various dictionaries of random matrices.
\end{minipage}
\end{center}

\section{Introduction}

Inferring reliable information from limited data is a key task in the
sciences. For example, identifying a channel operator from its
response to a limited number of test signals is a crucial step in
radar and communications engineering \cite{GR05,HS07,KP06,PDT91,PW06b,Sko80}. Here we consider the canonical setting where an operator is approximated by a linear map, that is, by a matrix $\op\in\CC^{n\times m}$. While it is clear that $\op$ is determined by its action on any $m$
vectors that span $\CC^m$, significantly fewer measurements may be
sufficient if {\it a-priori} information about the operator is at hand. For instance, one commonly considers the question whether a single test signal $\boldsymbol h$, referred to also as identifier, can be used to identify $\op$ from  $\op \boldsymbol{h}$. {\it A priori} information guaranteeing that such an $\boldsymbol h$ exists is generally deduced from physical considerations which may ensure that $\op$ can be
efficiently represented or approximated using relatively few basic
matrices from a known matrix dictionary.

In wireless communications (\cite{Cor01,GP07,Pae01} and references within) and sonar \cite{mi87,st99},
for example, the narrowband regime of a transmission channel
can generally be well approximated by a linear combination of a
small number of time-frequency shift matrices.
Signals travel from the source to the receiver along a number of
different paths, each of which can be modeled by a time shift (delay
dependent on the length of the path traveled) and a frequency shift
(Doppler effect caused by the motion of the transmitter, of the
receiver, and of reflecting objects) \cite{be63,GP07}. It is frequently
assumed, that the number of relevant (but unknown) paths, that is, in slightly simplified terms
the number of involved time-frequency shifts
is relatively small when compared to the symbol length. For example, for mobile communications the number of paths required to well approximate a channel in rural areas or typical urban regiments does not exceed 10  \cite[pages 266,283]{Pae01}, see also \cite{Cav02,Cor01}. 
In wireless communications the benefit of recovering the operator at
the receiver is clear. Knowledge of the operator is necessary to
invert it and to recover the information carrying channel input from
the channel output.

Complexity regularization has recently seen a resurgence of interest
in the signal processing community under the monikers {\it sparse
signal recovery} and {\it sparse
approximation}. In sparse signal recovery, one seeks the solution of an underdetermined
system of equations $\boldsymbol{Ax}=\boldsymbol{b}$, $\boldsymbol{A}\in\CC^{n\times N}$, $n<N$, with $\boldsymbol x$
having the fewest number of non-zero entries from all solutions of
$\boldsymbol{Ax}=\boldsymbol{b}$. We show in Section~\ref{sec:results-context} that the
identification of a matrix from its
action on a single test signal falls into the same setting as
sparse signal recovery when the matrix
is known to have a sparse representation. This observation
allows us to adopt efficient algorithms from sparse signal recovery for
the sparse matrix identification question.   Examples of applications include
the channel identification,  estimation, or sounding problem
described in part above, which also have been considered in the case
of time-invariant channels in \cite{CS05,CR02,HKP05}.
Numerical results based on Basis Pursuit
have been obtained for time-varying channels in \cite{SKPW05}. Further, the application of recovery methods of sparsely represented operators to radar measurements is discussed in \cite{HS07}.

In brief, the content of this paper is organized as follows. In Section~\ref{sec:results-context} we formalize the matrix identification
problem for matrices with sparse representations. We establish a
connection to the recovery problem of vectors with
sparse representations and state the main results that are proven
and discussed in greater detail in Section~\ref{Sec:Random} and
Section~\ref{sec:gabor}. In particular, we consider matrix ensembles
of random Gaussian or Bernoulli matrices as well as partial Fourier
matrices (Section~\ref{sec:subsectionRandom} and
Section~\ref{Sec:Random}).

In Section~\ref{sec:subsectionTimeFrequShifts} and
Section~\ref{sec:gabor} we consider matrix dictionaries of
time-frequency shift matrices which are of
particular interest due to their efficacy in approximating time-varying
transmission channels.  We would like to emphasize that the common framework
of the identification problem for matrices with a sparse representation
and the sparse signal recovery problem implies that the results achieved
on the recovery of matrices with a sparse representation in the dictionary
of time-frequency shift matrices are at the same time results for the
recovery of signals with a sparse representation in Gabor frames.

In Section \ref{sec:severalvectors} we briefly discuss the use
of several test vectors instead of just one, and comment on how this improves
corresponding recovery results.

We conclude with numerical experiments in Section~\ref{sec:numerics}. They verify
our main results concerning sparse representations with time-frequency shift matrices stated in Theorem~\ref{thm:bp_gabor_nlogn}, and show that the
precise recoverability
thresholds follow those proven for Gaussian random matrices in \cite{dota06};
that is, for matrices having a $k$-sparse representation we observe Basis
Pursuit to successfully recover the matrix from its action on a single
vector provided $k\le n/(2\log n)$.

\section{Main results and context}\label{sec:results-context}

Before comparing the matrix identification problem with sparse signal recovery, we
formalize the notion of a matrix having a $k$-sparse representation.

\begin{definition}\label{def:sparse}
A matrix $\op$ has a $k$-sparse representation
in the matrix dictionary $\BPsi\defleft\{\BPsi_j\}_{j=1}^N$ if
\begin{equation*}
\op=\sum_{j}x_j\BPsi_j\quad\quad\mbox{with}\quad
\| \boldsymbol x\|_0=k,
\end{equation*}
and $\|\boldsymbol x\|_0$ counts the number of non-zero entries in $\boldsymbol{x}$, that is $\|\boldsymbol x\|_0=|\supp \boldsymbol{x} |={ \rm cardinality} \{x_j: \, x_j\neq 0\}$.
\end{definition}

The set of elementary matrices comprising $\BPsi$ may form a basis for
$\CC^{n\times m}$ but it may as well only span a subspace of $\CC^{n \times m}$ and/or contain linearly dependent subsets.  In Definition~\ref{def:sparse} we place no restrictions on the
dictionary $\BPsi$.

Identification of matrices having a sparse representation
from their action on a
single vector (henceforth referred to simply as {\em sparse matrix identification}, which is not to be confused with the notion of sparse matrices in numerical analysis) can be formulated as sparse signal recovery problem
through the simple observation that the action of $\op$ on a test signal
$\boldsymbol{h}\in\CC^m$ can be expressed as
\begin{eqnarray}
\op \boldsymbol{h}&{=}&\Big(\sum_{j=1}^N x_j\BPsi_j\Big) \boldsymbol{h}=\sum_{j=1}^N x_j\big(\BPsi_j
\boldsymbol{h}\big) \label{eq:matrix} = \left(\BPsi_1 \boldsymbol{h}\,|\, \BPsi_2 \boldsymbol{h}\,|\, \ldots \,|\, \BPsi_N
\boldsymbol{h}\right)\boldsymbol{x}{\defright}(\BPsi \boldsymbol{h})\boldsymbol{x}
\end{eqnarray}
where $\boldsymbol{x}=(x_1,\; x_2,\; \cdots ,\; x_N)^T$ and
$ (\BPsi \boldsymbol{h})= (\BPsi_1 \boldsymbol{h}\,|\,\BPsi_2 \boldsymbol{h}\,|\,\ldots\,|\,\BPsi_N \boldsymbol{h} ).$

In classical sparse signal recovery the sparsest vector $\boldsymbol{x}$
satisfying $\boldsymbol{Ax}=\boldsymbol{b}$ is sought given $\boldsymbol{b}$ and $\boldsymbol{A}$; to identify the matrix
$\op$, $\op \boldsymbol{h}$ takes the place of $\boldsymbol{b}$ and the $j^{th}$ column of $\boldsymbol{A}$
is $\BPsi_j \boldsymbol{h}$ for $j=1,2,\ldots,N$.

As mentioned above, we note that in case of the $\BPsi_j$ being time-frequency shift matrices, the columns in $ \boldsymbol{A}=(\BPsi \boldsymbol{h})$ form a Gabor system with window $\boldsymbol{h}$ \cite{Chr03,gro01,lapfwa05}. Consequently, all our identifiability results concerning representations with time-frequency shift matrices are also results for the recovery of signals that are sparse in a Gabor system.

\begin{remark}\rm
Although sparse matrix identification
can be cast as sparse signal recovery, two important differences
should be noted.
\begin{itemize}
\item In most applications, sparse signal recovery is only of interest
for $k$-sparse vectors with $k<n$, as the linear dependence of the $N$ columns
of $\boldsymbol{A}\in \CC^{n\times N}$, $n<N$, implies that $n$-term solutions $\boldsymbol{x}$
for $\boldsymbol{Ax}=\boldsymbol{b}$ are never unique.  However, in some cases
an $n$-term solution might be of interest if there is no sparser solution of
$\boldsymbol{Ax}=\boldsymbol{b}$.    In contrast, the goal in sparse matrix
identification is not to represent $\boldsymbol{b}=\op \boldsymbol{h}$
efficiently, but to recover $\op$. The non-uniqueness of $n$-term solutions
to $(\BPsi \boldsymbol{h})\boldsymbol{x}=\op \boldsymbol{h}$ implies that there always exist
infinitely many $n-$sparse matrices
$\op^{\prime}$ consistent with the observations
$\op^{\prime}\boldsymbol{h}=\op \boldsymbol{h}$.
As such, the recovery of an
$n$-sparse $\boldsymbol{x}$ in the sparse matrix identification setting does
not give any information about the matrix to be identified, $\op$.
\item
In sparse signal recovery the columns of $\boldsymbol{A}$ are used to represent
or to approximate $\boldsymbol{b}$, whereas for sparse matrix identification the matrices $\BPsi_j$ are used
to represent or approximate $\op$.  However, unlike sparse signal recovery where the columns of $\boldsymbol{A}$ appear explicitly in the reconstruction, the $\BPsi_j$ do not appear explicitly when
sparse matrix identification is cast as sparse signal recovery (\ref{eq:matrix}); rather,
only the action of $\BPsi_j$ on the test vector $\boldsymbol{h}$ is utilized.
The test vector
$\boldsymbol{h}\in \CC^m$ has no analog in traditional sparse signal recovery,
and can be exploited in sparse matrix identification to design desirable
characteristics in $\BPsi_j \boldsymbol{h}$.  This design freedom is utilized extensively
in our main results concerning the matrix dictionary of time-frequency shifts,
Theorem~\ref{thm:bp_gabor_nlogn}.
\end{itemize}
\end{remark}

Note that the computational difficulty in sparse signal recovery, sparse approximation, and our formulation of sparse matrix identification arises from the fact that the support set of the non-zero entries in $\boldsymbol{x}$
is unknown.
While the direct solution of finding the sparsest representation of
$\op$ in the dictionary $\BPsi$
\begin{equation}\label{eq:l0}
\min \|\boldsymbol{x}'\|_{0} \quad\mbox{subject to}\quad (\BPsi \boldsymbol{h})
\boldsymbol{x}'=\op \boldsymbol{h},
\end{equation}
involves a combinatorial search of the support set
and is therefore computationally intractable, a number of
computationally efficient algorithms have been shown to recover the
sparsest solution if appropriate conditions are met.  We
concentrate here on recoverability conditions for the canonical
sparse signal recovery algorithm Basis Pursuit (BP) where the convex
problem
\begin{equation}\label{eq:l1}
\min \|\boldsymbol{x}'\|_{1} \quad\mbox{subject to}\quad (\BPsi \boldsymbol{h})
\boldsymbol{x}'=\op \boldsymbol{h},
\end{equation}
$\|\boldsymbol{x}\|_1=\sum_j |x_j|$, is solved as a proxy to (\ref{eq:l0}).

The convex program (\ref{eq:l1}) can be solved efficiently using well
established optimization algorithms for second-order cone
programming and linear programming \cite{bova04,dots06-1,bogokikslu07},
for complex and real valued systems, respectively.  We give theoretical and
numerical evidence for conditions where the solution of (\ref{eq:l1})
coincides exactly with that of (\ref{eq:l0}).
Many other algorithms may also be used as proxys for (\ref{eq:l0}),
including Orthogonal Matching Pursuit (OMP)
\cite{gitr05,kura06,tr04}, stagewise
orthogonal matching pursuit (StOMP) \cite{dots06}, and an algorithm
based upon error correcting codes \cite{akta07}--to name a few.
Our principal technical results in Section~\ref{subsec:gabor_coherence}
also give results for
OMP, but for conciseness we do not state them here, leaving them to
the interested reader.

In practice, the measured vector $\op \boldsymbol{h}$ will be contaminated by
noise, and, in addition, the operator $\op$ will not be strictly sparse,
but will instead be well approximated by a sparse representation;
in this case the minimization problem (\ref{eq:l1})
will be replaced by its well known variant 
\begin{equation}\label{eq:l1noise}
\min \|\boldsymbol{x}'\|_{1} \quad\mbox{subject to}\quad \|(\BPsi \boldsymbol{h})\boldsymbol{x}'-\op \boldsymbol{h}\|_2\leq \epsilon,
\end{equation}
where $\|\boldsymbol{z}\|_2=\sqrt{\sum_j |z_j|^2}$ as usual.

\subsection{Dictionaries of random matrices}\label{sec:subsectionRandom}

Many known results in sparse signal recovery, sparse approximations and their companion theory
of compressed sensing involve random matrices
\cite{badadewa06,cata06,do04,dota06,rascva06}. Based on these
results, we obtain recovery results for matrix dictionaries where
all its member matrices are chosen at random. From a practical point
of view such random matrix dictionaries do not seem to be useful in
the sparse matrix identification setting;
nevertheless, the statements give some insight into the sparse matrix
identification question as they give guidance in what kind of results to seek
 in the mathematical analysis of structured and more  application relevant  matrix dictionaries.

\begin{theorem}\label{thm:Random} Let $\boldsymbol{h}$ be a non-zero vector in $\RR^m$.
\begin{itemize}
\item[(a)]
Let all entries of the $N$ matrices $\BPsi_j \in \RR^{n\times m}$,
$j=1,\hdots,N$ be chosen independently according to a standard
normal distribution (Gaussian ensemble); or
\item[(b)]  let all entries of the $N$ matrices
$\BPsi_j \in \RR^{n\times m}$, $j=1,\hdots,N$ be independent
Bernoulli $\pm 1$ variables (Bernoulli ensemble).
\end{itemize}
Then there exists a positive constant $c$ so that for $\varepsilon>0$,
\[
k \leq c\,\frac{n}{\log\big(\frac{N}{n\varepsilon}\big)}
\]
implies that with probability of at least $1-\varepsilon$ all
matrices $\op$ having a $k$-sparse representation with respect to
$\BPsi = \{\BPsi_j\}$ can be recovered from $\op \boldsymbol{h}$ by Basis Pursuit
(\ref{eq:l1}).
\end{theorem}

Using Theorem~\ref{thm_bp_stab}, this recovery result can be made
stable under perturbation of $\op \boldsymbol{h}$ by noise, and also applies
when $\op$ is not exactly $k$-sparse,
but can be well approximated by a $k$-sparse operator.

Precise information on the constant $c$ will be given in
Section~\ref{Sec:Random}. In case of the Gaussian ensemble Donoho and Tanner
\cite{dota05,do05,do06-1,dota05-1,dota06} obtained sharp
thresholds separating regions in the ($k/n$, $n/N$) plane
where recovery holds or fails with high probability;
Section~\ref{subsec:gaussian} recounts these and additional results on
Gaussian systems.  Theorem~\ref{thm:Random}(b) is proven in
Section~\ref{subsec:bernoulli}, and similar results for certain
diagonal
matrices are proven in Section~\ref{subsec:fourier}.

\subsection{The dictionary of time-frequency shift
matrices}\label{sec:subsectionTimeFrequShifts}

As outlined in the introduction, the matrix dictionary of
time-frequency shifts  appears naturally in the channel
identification problem in wireless communications \cite{be63} or
sonar \cite{st99}. Due to physical
considerations wireless channels may indeed be modeled by sparse
linear combinations of time-frequency shifts
$\boldsymbol{M}_{\ell} \boldsymbol{T}_p$, where
the periodic translation operators $\boldsymbol{T}_p$ and
modulation operator $\boldsymbol{M}_\ell$ on $\CC^n$ are given by
\begin{equation}\label{eq:trans_mod}
(\boldsymbol{T}_p \boldsymbol{h})_q =h_{(p+q)\!\!\!\!\!\mod{n}},\;\; (\boldsymbol{M}_{\ell}\boldsymbol{h})_q \defleft e^{2\pi i\ell q/n} h_q.
\end{equation}
The system of time-frequency shifts,
\begin{equation}\label{eq:bg}
\BG = \{\boldsymbol{M}_\ell \boldsymbol{T}_p: \ell,p =0,\hdots, n{-}1\},
\end{equation}
forms a basis of $\CC^{n\times n}$ and for any non-zero $\boldsymbol{h}$,
the vector dictionary $\BG \boldsymbol{h}$ is a Gabor system
\cite{gro01,krpfra07,lapfwa05}. Below, we focus on the so-called
Alltop window $\boldsymbol{h}^A$ \cite{al80,hest03} with entries
\begin{equation}\label{eq:h1}
h^A_q \,\defleft\, \frac{1}{\sqrt{n}}
e^{2\pi i q^3/n}, \quad q=0,\hdots,n{-}1,
\end{equation}
and the randomly generated window $\boldsymbol{h}^R$ with entries
\begin{equation}\label{eq:h2}
h^R_q \,\defleft\, \frac{1}{\sqrt{n}} \epsilon_q, \quad q=0,\hdots,n{-}1,
\end{equation}
where the $\epsilon_q$ are independent and uniformly distributed
on the torus $\{z \in \CC, |z|=1\}$.

Invoking existing recovery results \cite{doelte06,grva06,tr04,tr06-3}
(see Theorems~\ref{thm:coh} and \ref{thm:stable:coh} below)
and our results on the coherence of Gabor systems
$\BG \boldsymbol{h}^A$ and  $\BG \boldsymbol{h}^R$ in Section~\ref{subsec:gabor_coherence}, see
Section~\ref{cor1}, we will obtain

\begin{theorem}\label{cor1}
\begin{itemize}
\item[(a)] Let $n$ be prime and $\boldsymbol{h}^A$ be the Alltop window
defined in (\ref{eq:h1}). If $k < \frac{\sqrt{n}+1}{2}$ then Basis Pursuit
recovers from $\op \boldsymbol{h}^A$ all matrices $\op\in\CC^{n\times n}$ having a $k$-sparse
representation, $\op = \sum_{(p,\ell) \in \Lambda} x_{p\ell} \boldsymbol{M}_\ell \boldsymbol{T}_p$, $|\Lambda|=k$, with respect to the time-frequency shift dictionary $\BG$ given in (\ref{eq:bg}).
\item[(b)] Let $n$ be even and choose $\boldsymbol{h}^R$ to be the random unimodular
window in (\ref{eq:h2}). Let $t > 0$ and suppose
\begin{equation}\label{cond:k}
k \leq \frac{1}{4} \sqrt{\frac{n}{2 \log n + \log 4 + t}} + \frac{1}{2}\,.
\end{equation}
 Then with probability of at least
$1-e^{-t}$ Basis Pursuit recovers from $\op \boldsymbol{h}^R$ all matrices $\op \in
\CC^{n \times n}$ having a $k$-sparse representation with respect to the
time-frequency shift dictionary $\BG$ given in (\ref{eq:bg}).
\end{itemize}
\end{theorem}
A slight variation of part (b) also holds for $n$ odd, but is omitted for conciseness. Further note that Theorem~\ref{cor1} also holds with Basis Pursuit literally being replaced by Orthogonal Matching Pursuit \cite{tr04}. Moreover, Theorem~\ref{thm:stable:coh} shows that recovery is stable under perturbation of $\op \boldsymbol{h}^A$ and $\op \boldsymbol{h}^R$ by noise.

In contrast with Theorem~\ref{thm:Random} for random matrices, where $k$ is allowed
to be of order ${\cal O}(n / \log n)$,
Theorem~\ref{cor1} requires $k$ to be of order  $\sqrt{n}$ or
$\sqrt{n/\log n}$.  Substantially larger order  thresholds,  ${\cal O}(n / \log n)$  for  $\boldsymbol{h}^A$ and ${\cal O}(n / \log^{2}(n))$ for $\boldsymbol{h}^R$,  are also possible  to identify a matrix $\op$ which is the linear combination of a small number of time-frequency shift matrices. However, this larger regime of successful recovery necessitates passing from a worst case analysis for sparse $\op$ to an average
case analysis in the sense that the coefficient vector $\boldsymbol{x}$ is chosen
at random.  Theorem~\ref{thm:bp_gabor_nlogn} will follow from recent
work by Tropp, \cite{tr06-2}, and our coherence results in
Section~\ref{subsec:gabor_coherence}, see
Section~\ref{subsec:thm5.2}.

\begin{theorem}\label{thm:bp_gabor_nlogn}
Let $k \geq 3$ and let $\Lambda$ be chosen uniformly at random
among all subsets of $\{0,\hdots,n{-}1 \}^2$ of cardinality $k$. Suppose
further that $\boldsymbol{x} \in \C^n$ has support $\Lambda$ with random phases
$(\sgn(x_{\ell p}))_{(\ell,p)\in \Lambda}$ that are independent and
uniformly distributed on the torus $\{z, |z|=1\}$. Let
\[ \op =
\sum_{(\ell,p) \in \Lambda} x_{\ell p} \boldsymbol{M}_\ell \boldsymbol{T}_p.
\]
\begin{itemize}
\item[(a)]
Let $n$ be prime and choose the Alltop window $\boldsymbol{h}^A$ from
(\ref{eq:h1}). Assume that for $\epsilon > 0$
\begin{equation}\label{kcond:a:1}
k \leq \frac{n}{8 \log(2n^2/\epsilon)}
\end{equation}
and
\begin{eqnarray}
s&:=& \frac{1}{144}\left(e^{-1/4}/2 - \frac{2k}{n}\right)^2 \frac{n}{k \log(k/2+1)}\geq 1
\label{cond:s1}
\end{eqnarray}
Then with probability at least
\[
1 - (\epsilon + (k/2)^{-s})
\]
Basis Pursuit (\ref{eq:l1}) recovers $\op$ from $\op \boldsymbol{h}^A$.
\item[(b)]
Let $n$ be an even number and choose the random window $\boldsymbol{h}^R$ from
(\ref{eq:h2}).
Assume
\begin{equation*}
k \leq \frac{n}{32(\sigma + 2) \log(n) \log(2n^2/\epsilon)}
\end{equation*}
for some $\sigma > 0$ and
\begin{eqnarray*}
 s &:=&   \frac{1}{576(\sigma + 2)}\left(e^{-1/4}/2 - \frac{2k}{n}\right)^2\cdot \ \frac{n}{k \log(k/2+1)} \geq 1 
\end{eqnarray*}
Then with probability at least
\[
1 - (\epsilon + 4n^{-\sigma}  + (k/2)^{-s})
\]
Basis Pursuit (\ref{eq:l1}) recovers $\op$ from $\op \boldsymbol{h}^R$. (A
similar result also holds for $n$ odd.)
\end{itemize}
\end{theorem}

In simple terms, Theorem~\ref{thm:bp_gabor_nlogn} states that $\op$ can be
recovered from $\op \boldsymbol{h}^A$ or $\op \boldsymbol{h}^R$ with {\em high probability} $1-\varepsilon$
provided that the sparsity of $\op$ satisfies
$k \leq C_\varepsilon n /\log n$ in case of $\boldsymbol{h}^A$ and
$k \leq C'_\varepsilon n/\log(n)^2$ in case of $\boldsymbol{h}^R$.

In Section~\ref{subsec:cor2.6} we use a simple argument from time-frequency analysis to obtain
\begin{corollary}\label{cor:Fourier} Theorems~\ref{cor1},~\ref{thm:bp_gabor_nlogn},
and~\ref{lem:mu_prob}, also hold with the
windows $\boldsymbol{h}^A$ and $\boldsymbol{h}^R$ replaced by their Fourier transforms $\widehat{\boldsymbol{h}^A}$ and
$\widehat{\boldsymbol{h}^R}$, with entries defined as $\widehat{h}_j \,=\, \frac{1}{\sqrt{n}}
\sum_{j=0}^{n{-}1}h_q e^{2\pi i jq/n}$.
\end{corollary}

\section{Tools in sparse signal recovery}\label{sec:sqrtn}

It was shown in (\ref{eq:matrix}) that for any test signal $\boldsymbol{h}$, we
have $\op \boldsymbol{h}=(\BPsi \boldsymbol{h}) \boldsymbol{x} $ where $\boldsymbol{x}$ is the sparse coefficient vector
of $\op$. This observation links the sparse matrix identification question
with sparse signal recovery where one seeks the sparsest solution (\ref{eq:l0})
to the underdetermined system $\boldsymbol{Ax}=\boldsymbol{b}$; in the sparse matrix identification
setting $(\BPsi \boldsymbol{h}) \,=\, (\BPsi_1 \boldsymbol{h}\,|\, \BPsi_1 \boldsymbol{h}\,|\,\ldots \,|\,\BPsi_N \boldsymbol{h})$
takes the place of $\boldsymbol{A}$ and $\op \boldsymbol{h}$ the place of $\boldsymbol{b}$.
In contrast to sparse approximation, where the dictionary $\boldsymbol{A}$
is usually fixed, for sparse matrix identification we have the additional
freedom of designing the test signal $\boldsymbol{h}$ in order for $(\BPsi \boldsymbol{h})$ to have desirable properties.

Let us shortly recall known results in sparse signal recovery and sparse approximation that we
apply to the sparse matrix identification question.  In Section~\ref{subsec:coherence}
we review the notion
of coherence (\ref{eq:coherence}) and its implications for sparse
signal recovery and approximation using Basis Pursuit, (\ref{eq:l1}) and (\ref{eq:l1noise}),
as well as Orthogonal Matching Pursuit.  In Section~\ref{subsec:rip}
we review the restricted isometry property, allowing for
improved recoverability
results for Basis Pursuit. 

\subsection{Coherence}\label{subsec:coherence}

The recoverability properties of sparse signal recovery algorithms
for an underdetermined system $\boldsymbol{Ax}=\boldsymbol{b}$ is
often measured by the coherence of $\boldsymbol{A}$,
\begin{equation}\label{eq:coherence}
\mu \,\defleft\, \max_{r \neq s} |\langle \boldsymbol{a}_r,\boldsymbol{a}_s\rangle|,
\end{equation}
where $\boldsymbol{a}_r$ is the $r^{th}$ column of $\boldsymbol{A}$ and $\|\boldsymbol{a}_r\|_2=1$ for all $r$.

\begin{theorem}[Tropp \cite{tr04}; Donoho, Elad \cite{doel02}]\label{thm:coh}
Let $\boldsymbol{A}$ be a unit norm dictionary with coherence $\mu$.
If
\begin{equation*}
(2k-1) \mu < 1
\end{equation*}
then Basis Pursuit (as well as Orthogonal Matching Pursuit) recovers all $k$-sparse vectors $\boldsymbol{x}$ from
$\boldsymbol{b}=\boldsymbol{Ax}$.
\end{theorem}

Recovery is also stable under perturbation by noise when Basis Pursuit
(\ref{eq:l1}) is replaced with (\ref{eq:l1noise}).


\begin{theorem}[Donoho et al. \cite{doelte06}, Theorem 3.1]\label{thm:stable:coh}
Let $\boldsymbol{A}$, $\mu$ be as above and suppose that $(4k-1)\mu < 1$. Assume that $\boldsymbol{x}$ is $k$-sparse and we
have perturbed observations $\boldsymbol{b} = \boldsymbol{Ax} + \boldsymbol{z}$ with $\|\boldsymbol{z}\|_2 \leq \epsilon$.
Then the solution $\boldsymbol{x}^\#$ of the Basis Pursuit variant
\begin{equation}\label{BP:var}
\min \|\boldsymbol{x}'\|_1 \quad \mbox{\rm subject to} \quad \|\boldsymbol{Ax}' -\boldsymbol{b} \|_2 \leq \delta
\nonumber
\end{equation}
satisfies
\[
\|\boldsymbol{x}^\# - \boldsymbol{x}\|_2^2 \leq \frac{(\epsilon + \delta)^2}{1-\mu(4k-1)}.
\]
\end{theorem}

Theorems~\ref{thm:coh} and~\ref{thm:stable:coh} ensure that the solutions
of (\ref{eq:l1}) and (\ref{eq:l1noise}) correspond (exactly and approximately,
respectively) to the solution of (\ref{eq:l0}) for {\em all} $k$-sparse
$\boldsymbol{x}$.  For a broad class of dictionaries the coherence is of
order ${\cal O}(1/\sqrt{n})$,
see Sections 4 and 5 for random and Gabor dictionaries, respectively.
Hence, Theorems~\ref{thm:coh} and \ref{thm:stable:coh}
ensure (stable) recovery provided $k = {\cal O}(\sqrt{n})$.

In contrast to these ${\cal O}(\sqrt{n})$ thresholds, which are
valid for all $\boldsymbol{x}$,
Tropp \cite{tr06-2} developed a general framework for the analysis
of Basis Pursuit (\ref{eq:l1}), which is still based on the
coherence of a general dictionary, but shows that (\ref{eq:l1}) is often
successful for substantially larger $k$ than those considered in
Theorems~\ref{thm:coh}
and \ref{thm:stable:coh}.  This comes, however, at the cost
of assuming a random model on the sparse signal to be recovered. It
allows us to prove order ${\cal O}(n/ \log n)$ for $\boldsymbol{h}^A$ and  ${\cal O}(n/\log(n)^2)$ for $\boldsymbol{h}^R$
recoverability result for the
time-frequency-shift dictionary, Theorem~\ref{thm:bp_gabor_nlogn}.
We state the results of Tropp, where $\|\,\cdot\,\|_{2,2}$ denotes the operator norm given by $\| \boldsymbol{A} \|_{2,2}=\sup_{\|\boldsymbol{x}\|_2=1} \|\boldsymbol{Ax}\|_2$, and  $\boldsymbol{A}_\Lambda$ is the
restriction of a matrix $\boldsymbol{A}$ to the columns indexed by $\Lambda$.


\begin{theorem}[Tropp \cite{tr06-2}, Theorem 12]\label{thm:Tropp}
Let $\boldsymbol{A}$ be an $n \times N$ vector dictionary with unit norm columns
and coherence $\mu$. Suppose that $\Lambda$ is selected uniformly at
random among all subsets of $\{1,\hdots,N\}$ of size $k \geq 3$. Let
$s \geq 1$. Then
\begin{equation}\label{cond:Tropp1}
\sqrt{144 s \mu^2 k \log (k/2+1)} + \frac{2k}{N} \|\boldsymbol{A}\|_{2,2}^2 \leq e^{-1/4} \delta
\end{equation}
implies
\[
\P\left(\|\boldsymbol{A}_\Lambda^* \boldsymbol{A}_\Lambda - \Id\|_{2,2} \geq \delta\right) \leq
(k/2)^{-s}.
\]
\end{theorem}

\begin{theorem}[Tropp \cite{tr06-2}, Theorem 13]\label{thm:Tropp2}
Let $\boldsymbol{A}$ be an $n \times N$ dictionary with coherence $\mu$. Suppose
$\Lambda \subseteq \{1,\hdots,N\}$ of cardinality $k$ ($|\Lambda|=k$)
is such that
\[
\|\boldsymbol{A}_{\Lambda}^* \boldsymbol{A}_{\Lambda} - \Id\|_{2,2} \leq 1/2.
\]
Suppose that
$\boldsymbol{x} \in \C^N$ has support $\Lambda$ with random phases $\sgn(x_r)$, $r\in T$,
that are independent and uniformly
distributed on the torus $\{z,|z|=1\}$. Then with probability at
least $1-2Ne^{-1/(8\mu^2 k)}$ the sparse vector $\boldsymbol{x}$ can be recovered from $\boldsymbol{b} = \boldsymbol{Ax}$
by Basis Pursuit.
\end{theorem}

\subsection{Restricted isometry property}\label{subsec:rip}


Cand{\`e}s, Romberg and Tao introduced
the Restricted Isometry Property (RIP) which is an alternative
perspective to coherence \cite{carota06-1,cata06}.

\begin{definition}
\label{def:rip}
Let $\boldsymbol{A} \in \CC^{n\times N}$ and $k < n$.
The restricted isometry constant $\delta_k = \delta_k(\boldsymbol{A})$ is the smallest
number such that 
\[
(1-\delta_k) \|\boldsymbol{x}\|_2^2 \leq \|\boldsymbol{A x}\|_2^2 \leq (1+\delta_k) \|\boldsymbol{x}\|_2^2
\]
for all $k$-sparse $\boldsymbol{x}$.
\end{definition}

$\boldsymbol{A}$ is said to satisfy the restricted isometry property
if it has small isometry constants, say $\delta_k < 1/2$; such
matrices allow stable sparse recovery by Basis Pursuit.

\begin{theorem}[Cand{\`e}s, Romberg and Tao \cite{carota06-1}]\label{thm_bp_stab}
Assume that the restricted isometry constants of $\boldsymbol{A}$ satisfy
\[
\delta_{3k} + 3 \delta_{4k} < 2.
\]
Let $\boldsymbol{x} \in \C^N$ and assume we have noisy data
$\boldsymbol{y}=\boldsymbol{Ax}+\boldsymbol{\eta}$ with $\|\boldsymbol{\eta}\|_2 \leq \epsilon$.
Denote by $\boldsymbol{x}^k$ the truncated vector corresponding
to the $k$ largest absolute values of $\boldsymbol{x}$. Then the solution $\boldsymbol{x}^\#$
of (\ref{eq:l1noise}) satisfies
\begin{equation*}
\|\boldsymbol{x}^\# - \boldsymbol{x}\|_2 \,\leq\, C_1 \epsilon + C_2 \frac{\|\boldsymbol{x}-\boldsymbol{x}^k\|_1}{\sqrt{k}}.
\end{equation*}
The constants $C_1$ and $C_2$ depend only on $\delta_{3k}$ and $\delta_{4k}$.
\end{theorem}

Note that for $\boldsymbol{x}$ $k$-sparse and noise level $\epsilon = 0$,
Theorem~\ref{thm_bp_stab} guarantees exact recovery
of $\boldsymbol{x}$ by (\ref{eq:l1}).

\section{Random matrices}
\label{Sec:Random}

Many of the recent results in sparse signal recovery with
recoverability thresholds for $k\le C n/\log n$ either
assume that $\boldsymbol{A}$ is a random Gaussian or Bernoulli matrix
\cite{badadewa06,cata06,do04,rascva06}, or partial random Fourier
matrix \cite{carota06,kura06,ra05-7,ra06,ru06-1}. Recoverability
results in these cases can be obtained by establishing the restricted
isometry property, see Definition~\ref{def:rip}, or through a
careful analysis of the geometric structure of the convex hull
associated with the columns of $\boldsymbol{A}$ \cite{dota05,do05,do06-1,dota05-1,dota06}. We
apply these results to the  matrix identification problem when the
matrix has a sparse representation in terms of certain random
matrices.

\subsection{Gaussian matrix ensemble}\label{subsec:gaussian}

Assume all entries of the $N$ matrices $\BPsi_j \in \R^{n\times m}$ in $\BPsi$
are independent standard Gaussian random variables and $\boldsymbol{h}$ is an arbitrary
non-zero vector in $\R^m$.
Then the entries of the dictionary $\boldsymbol{A} = (\BPsi \boldsymbol{h}) \in
\R^{n \times N}$ whose columns are given by $\BPsi_j \boldsymbol{h}$,
$j=1,\hdots,N$, are jointly independent and of the form
$Z = \sum_{\ell=1}^n g_\ell h_\ell$
where the $g_\ell$ are independent standard Gaussian random
variables. By rotational invariance of the distribution of the
Gaussian vector $(g_1,\hdots,g_n)$ the random variable $Z$ has the
same distribution as $\|\boldsymbol{h}\|_2 g$ where $g$ is a (scalar-valued)
standard Gaussian. Hence, the dictionary $(\BPsi \boldsymbol{h})$ has the same
distribution as
$\|\boldsymbol{h}\|_2 \boldsymbol{A} \in \R^{n \times N},$
where $\boldsymbol{A}$ is a random matrix whose entries are independent standard
Gaussians. Thus, the existing literature in sparse approximation
concerning Gaussian matrices applies, see for instance
\cite{badadewa06,cata06,do04,dota06,rascva06} and additional
results discussed in the remainder of this section.

In particular, the restricted isometry property ensures
stable recovery with probability at
least $1-\varepsilon$ provided \cite{badadewa06,cata06,rascva06}
\begin{equation}\label{eq:rip_small}
k \leq c \frac{n}{\log(\frac{N}{n\varepsilon})}.
\end{equation}
Hence, by Theorem~\ref{thm_bp_stab} we have stable recovery by
(\ref{eq:l1noise})
in this regime and the statement of Theorem \ref{thm:Random}(a) follows.

The work of Donoho and Tanner \cite{do05,do06-1} actually
allows for a stronger statement than (\ref{eq:rip_small})
in the context of noise-free and
exact $k$-sparse vectors $\boldsymbol{x}$. A simple version of their results says that
most $k$-sparse $\op$ can be recovered with high probability by Basis Pursuit
provided $k \leq \frac{n}{2\log(N/n)}$. For details we refer
to \cite{do05,do06-1}, and for extension to the noisy setting
to Wainwright's work \cite{wa06}.

\subsection{Bernoulli matrix ensemble}\label{subsec:bernoulli}

The recoverability results for Bernoulli matrices in Theorem~\ref{thm:Random}(b) are based on establishing the restricted isometry property given in Definition~\ref{def:rip}.

To this end,  we assume that
the entries of the $N$ matrices $\BPsi_j \in \R^{n \times m}$ in $\BPsi$  are selected as independent $\pm 1$ Bernoulli variables, that is, $+1$ or $-1$ with equal probability, and let $\boldsymbol{h}$ be an
arbitrary non-zero vector. Then an entry of the dictionary $\boldsymbol{A} =
(\BPsi \boldsymbol{h})$ is given by
\begin{equation}\label{def:Bernoulli}
a_{pq} = \sum_{\ell=1}^n \epsilon_\ell^{pq} h_\ell,\quad
p=1,\hdots,m,~q=1,\hdots,N,
\end{equation}
where the $\epsilon_\ell^{pq}$ are independent Bernoulli variables,
that is, the $a_{pq}$ are independent Rademacher series \cite{leta91}.
Theorem~\ref{thm:bernoulli_rip}
shows that the matrix $\boldsymbol{A}$ has the restricted isometry property
with high probability for
sparsities $k$ that are nearly linear in $m$. Hence, by
Theorem~\ref{thm_bp_stab}, for an arbitrary non-zero choice of $\boldsymbol{h}$
we can recover any $\op$ having a $k$-sparse representation in terms
of random Bernoulli matrices from the action of $\op \boldsymbol{h}$ through
Basis Pursuit (\ref{eq:l1}).

\begin{theorem}\label{thm:bernoulli_rip}
Let $\boldsymbol{h} \in \R^m$ be normalized by $\|\boldsymbol{h}\|_2 = 1/\sqrt{m}$. Let
$\boldsymbol{A}$ be the random matrix with entries defined in (\ref{def:Bernoulli}).
Assume $\delta \in (0,1)$ and $t > 0$. If
\begin{equation} \label{cond:Bernoulli}
n \geq C_1 \delta^{-2}(k \log(N/k) + \log(2e+24e/\delta) + t).
\end{equation}
Then with probability at least $1- e^{-t}$ the restricted isometry
property is satisfied, that is, for all $\Lambda \subset
\{1,\hdots,N\}$ of cardinality at most $k$ it holds that
\[
(1-\delta) \|\boldsymbol{x} \|_2^2 \leq \|\boldsymbol{A x}\|_2^2
\leq (1+\delta) \|\boldsymbol{x}\|_2^2
\]
for all $\boldsymbol{x}$ supported on $\Lambda$. The constant satisfies
$C_1 \leq 23.15$.
\end{theorem}
\begin{proof}
Let $\boldsymbol{v} \in \R^N$ be an arbitrary vector. We form the inner product
of a row of $\boldsymbol{A}$ with $\boldsymbol{v}$,
\[
X_p \,\defleft\, \sum_{q=1}^n a_{pq} v_q \,=\, \sum_{q=1}^N
\sum_{\ell=1}^n \epsilon_\ell^{pq} h_\ell v_q.
\]
By independence of the $\epsilon_\ell^{pq}$, the $X_p$ are similarly
independent. By Khintchine's inequality the even moments of
$X$ can be estimated by the moments of a standard Gaussian variable
$g$ \cite{leta91,pesh95}
\begin{eqnarray*}
  \E[|X_p|^{2z}] &\leq& \|\boldsymbol{v}\|_2 \|\boldsymbol{h}\|_2 \frac{(2z)!}{2^z z!}= \|\boldsymbol{v}\|_2 \|\boldsymbol{h}\|_2 \E[|g|^{2z}],
\quad z \in \N.
\end{eqnarray*}

Following Lemma 5 and the proof of Lemma 6 in
\cite{ac01} this implies the concentration inequality,
\begin{eqnarray*}
  \P(|\|\boldsymbol{A v}\|_2^2 - \|\boldsymbol{v}\|_2^2| \geq \epsilon \|\boldsymbol{v}\|_2^2)
\leq 2 \exp\left(-\tfrac{n}{2}(\epsilon^2/2 - \epsilon^3/3)\right).
\end{eqnarray*}
By Theorem 2.2 in \cite{rascva06}, see also
Theorem 5.2 in \cite{badadewa06}, this implies that
the restricted isometry property holds
under the stated condition on $n$. The estimate of the constant $C_1$ follows
 from \cite[Theorem 2.2]{rascva06} as well.
\end{proof}

Note that for fixed $\delta$ and $t$ condition (\ref{cond:Bernoulli}) can be
rewritten as
\[
k \leq c n/\log(N/k)
\]
for some constant $c$.

Combining Theorems~\ref{thm_bp_stab} and \ref{thm:bernoulli_rip}
yields Theorem~\ref{thm:Random}(b).

\subsection{Diagonal matrices}\label{subsec:fourier}

Diagonal matrices act as multiplication operators on $\C^n$. Using a Fourier expansion of the diagonal, we observe that any diagonal matrix can be expressed as linear combination of modulation operators $\boldsymbol{M}_\ell \in \C^{n\times n}$, $\ell=0,\hdots,n{-}1$,
defined in (\ref{eq:trans_mod}). We now consider the case that only a small number of components of the output of a diagonal operator $\op$ can be measured; the assumption that $\op$ is  sparse in the dictionary of modulation operators shall be used to recover $\op$ from these components.

To this end, let $\Omega$ be a subset of $\{0,\hdots,n{-}1\}$ of cardinality
$m$ and denote by $\boldsymbol{M}_\ell^\Omega \in \C^{m \times m}$ the submatrix
of $\boldsymbol{M}_\ell$ with columns and rows restricted to the index set
$\Omega$. Let
\[
\BPsi^\Omega = \{\boldsymbol{M}^\Omega_\ell, \ell=0,\hdots,n{-}1\}
\]
and $\boldsymbol{h}={\mathbf 1} = (1,\hdots,1)^T$. If $\op^\Omega =
\sum_{\ell=0}^{n{-}1} x_\ell \boldsymbol{M}_\ell^\Omega$ then $\op^\Omega {\mathbf
1}$ coincides with the restriction of $\op {\mathbf 1} =
\sum_{\ell=0}^{n{-}1} x_\ell \boldsymbol{M}_\ell {\mathbf 1}$ to the indices in
$\Omega$.

The matrix $\boldsymbol{A}$ whose columns are the elements of the dictionary
$(\BPsi^\Omega {\mathbf 1})  = \{\boldsymbol{M}_\ell^\Omega {\mathbf 1} , \ell=0,\hdots,n{-}1\}$
is precisely a row submatrix of the Fourier matrix,
\[
\boldsymbol{A} \,=\, \boldsymbol{A}^\Omega \,=\, (e^{2\pi i r \ell})_{r \in \Omega,\ell=0,\hdots,n{-}1} \in \C^{m \times n}.
\]
If the subset $\Omega$ is chosen uniformly at random among all
subsets of size $m$ then $\boldsymbol{A}^\Omega$ is a random matrix. This random
partial Fourier matrix was studied in \cite{carota06,cata06,ru06-1},
see also \cite{ra05-7} for a slight variation. Indeed, under the
condition
\[
k \leq c\, \frac{m}{\log^4(n)\log(\varepsilon^{-1})}
\]
the restricted isometry property holds with probability
at least $1-\varepsilon$
\cite{ru06-1} and by Theorem~\ref{thm_bp_stab} we obtain stable
recovery of all matrices having a sparse representation in terms of
$\BPsi^\Omega$.

\section{Time-frequency shift dictionaries}
\label{sec:gabor}




In this section we establish coherence results for the dictionary of
time-frequency shift matrices and prove Theorems~\ref{cor1}
and~\ref{thm:bp_gabor_nlogn}.

\subsection{Coherence for the time-frequency shift dictionary}\label{subsec:gabor_coherence}
We apply known recovery results
\cite{doelte06,grva06,tr04,tr06-3,tr06-2} for dictionaries with
small coherence (\ref{eq:coherence}).  Assuming $\|\boldsymbol{h}\|_2=1$,
the coherence, (\ref{eq:coherence}), of Gabor systems is
\begin{equation}\label{def:coherence1}
\mu \,\defleft\, \max_{(\ell,p) \neq (\ell',p')} |\langle \boldsymbol{M}_\ell \boldsymbol{T}_p h,
\boldsymbol{M}_{\ell'} \boldsymbol{T}_{p'} \boldsymbol{h}\rangle|.
\end{equation}
Based on results by Alltop in \cite{al80}, Strohmer and Heath showed
in \cite{hest03} that the coherence (\ref{def:coherence1}) of $\BG \boldsymbol{h}^A$
given in (\ref{eq:h1}) satisfies
\begin{equation}\label{coh:h1}
\mu \,=\, \frac{1}{\sqrt{n}}
\end{equation}
for $n$ prime.
This is almost optimal since the general lower bound in \cite{hest03}
for the coherence of frames with $n^2$ elements in $\C^n$  yields
$\mu \geq \frac{1}{\sqrt{n+1}}$.

Unfortunately, the coherence (\ref{def:coherence1}) of $\boldsymbol{h}^A$
applies only for $n$ prime.
For arbitrary $n$ we consider the random window $\boldsymbol{h}^R$.

\begin{theorem}\label{lem:mu_prob} Let $n \in \N$ and
choose a random window $\boldsymbol{h}^R$ with entries
\begin{equation*}
h^R_q \,\defleft\, \frac{1}{\sqrt{n}} \epsilon_q, \quad q=0,\hdots,n{-}1,
\end{equation*}
where the $\epsilon_q$ are independent and uniformly distributed on
the torus $\{z \in \CC, |z|=1\}$. Let $\mu$ be the coherence of the
associated Gabor dictionary (\ref{def:coherence1}),
then for $\alpha > 0$ and $n$ even,
\begin{equation*}
\P\big( \mu \geq \frac{\alpha}{\sqrt{n}}\big) \leq 4 n(n{-}1) e^{-\alpha^2/4},
\end{equation*}
while for $n$ odd,
\begin{equation}\label{eq:p2}
\P\big(\mu \geq \frac{\alpha}{\sqrt{n}}\big) \leq
2n(n{-}1)\left(e^{-\frac{n{-}1}{n}\alpha^2/4} + e^{-\frac{n+1}{n}
\alpha^2/4}\right).
\end{equation}
\end{theorem}

Up to the constant factor $\alpha$, the coherence
in Theorem~\ref{lem:mu_prob} comes close to
the lower bound $\mu \geq \frac{1}{\sqrt{n+1}}$ with high
probability.  Theorems~\ref{cor1} and~\ref{thm:bp_gabor_nlogn} will follow from
these order ${\cal O}(1/\sqrt{n})$ coherence results in this section
and the Theorems~\ref{thm:coh} and~\ref{thm:stable:coh} of
\cite{doelte06,grva06,tr04,tr06-3} and Theorems~\ref{thm:Tropp}
and~\ref{thm:Tropp2} of Tropp \cite{tr06-2} respectively.

%

{\it Proof of Theorem~\ref{lem:mu_prob}.}
The technical details for $n$ even and odd are slightly different, for
conciseness we only state the proof for $n$ even, and outline the proof
for $n$ odd.

A direct computation shows that
\[
|\langle \boldsymbol{M}_{\ell'} \boldsymbol{T}_{p'} \boldsymbol{h}^R, \boldsymbol{M}_\ell \boldsymbol{T}_p \boldsymbol{h}^R \rangle| = |\langle
\boldsymbol{M}_{\ell - \ell'} \boldsymbol{T}_{p-p'} \boldsymbol{h}^R,\boldsymbol{h}^R\rangle|
\]
and, therefore, it suffices to consider $\langle \boldsymbol{M}_{\ell} \boldsymbol{T}_p \boldsymbol{h}^R,\boldsymbol{h}^R\rangle$,
$\ell,p = 0,\hdots,n{-}1$; furthermore, as $\langle \boldsymbol{M}_\ell
\boldsymbol{h}^R,\boldsymbol{h}^R\rangle =\langle \boldsymbol{M}_\ell {\mathbf 1},|\boldsymbol{h}^R|^2 \rangle= 0$ for $\ell\neq 0$, we consider only the case $p \neq 0$.

Writing $\epsilon_q = e^{2\pi i y_q}$
with $y_q \in [0,1)$ we obtain
\begin{eqnarray*}
  \langle \boldsymbol{M}_\ell \boldsymbol{T}_p \boldsymbol{h}^R,\boldsymbol{h}^R\rangle &=& \frac{1}{n} \sum_{q=0}^{n{-}1}
e^{2\pi i \frac{q \ell}{n}} \epsilon_{q-p} \overline{\epsilon_q}= \frac{1}{n} \sum_{q=0}^{n{-}1} e^{2\pi i \left(y_{q-p} - y_q +
\frac{q\ell}{n} \right)},
\end{eqnarray*}
where $\epsilon_{q-p}= \epsilon_{n+q-p}$ if $q-p < 0$, that is, the
indices are understood modulo $n$. Set
\[
\delta_q^{(p,\ell)} \,\defleft\, e^{2\pi i \left(y_{q-p} - y_q +
\frac{q\ell}{n} \right)},
\]
and note that $\delta_q^{(p,\ell)}$ is uniformly
distributed on the torus $\T$. However, the $\delta_q^{(p,\ell)}$,
$q=1,\hdots,n$, are no longer jointly independent. But nevertheless,
as we demonstrate in the following,
we can split all variables into two subsets of independent variables.

If $p=1$, $p=n{-}1$,
or if neither $p$ nor $n-p$ divide $n$, then the $n/2$ random variables
$\epsilon_0 \overline{\epsilon_p}, \epsilon_{p} \overline{\epsilon_{2p}},
\hdots, \epsilon_{p (n/2-1)} \overline{\epsilon_{p n/2}}$ are jointly
independent, as well as the
remaining $n/2$ variables
$\epsilon_{pn/2} \overline{\epsilon_{p (n/2+1)}},
\hdots,  \epsilon_{p (n{-}1)} \overline{\epsilon_{0}}$.  The indices are again
understood modulo $n$. If $p\geq 2$ or $n-p \geq 2$
divides $n$, then we form the $p$ random vectors
\begin{align}
\boldsymbol{Y}_1 \defleft & (\epsilon_0 \overline{\epsilon_p},  \epsilon_{p} \overline{\epsilon_{2p}},  \hdots,  \epsilon_{n-p} \overline{\epsilon_{0}}),\notag\\
\boldsymbol{Y}_2 \defleft & (\epsilon_1 \overline{\epsilon_{p+1}},  \epsilon_{p+1} \overline{\epsilon_{2p+1}},  \hdots,  \epsilon_{n-p+1} \overline{\epsilon_{1}}),\notag\\
& \vdots\notag \\
\boldsymbol{Y}_p \defleft &(\epsilon_{p-1} \overline{\epsilon_{2p-1}},  \epsilon_{2p-1}
\overline{\epsilon_{3p-1}},  \hdots,  \epsilon_{n{-}1} \overline{\epsilon_{p-1}}).\notag
\end{align}
These vectors are jointly independent. Moreover, $p \leq n/2$ allows partitioning the
entries of a single vector $\boldsymbol{Y}$ into two sets $\Lambda_p^1$ and $\Lambda_p^2$
with $|\Lambda_p^1| , |\Lambda_p^2| \geq 1$ and the elements of each set are jointly independent.
Indeed, this can be seen by forming subsets of two {\em adjacent elements}
of the form
$\{\epsilon_{k+jp} \overline{\epsilon_{k+(j+1)p}},
\epsilon_{k + (j+1)p} \overline{\epsilon_{k+(j+2)p}}\}$
with possibly a remaining single element subset.
Then all subsets are jointly independent and the two elements
inside a subset are independent as well.

Now by forming unions $\cup_{i=1}^p \Lambda_i^1$ and $\cup_{i=1}^p \Lambda_i^2$
we can always partition the index set $\{0,\hdots,n{-}1\}$
into two subsets $\Lambda_1$, $\Lambda_2 \subset \{0,\hdots,n{-}1\}$
with $|\Lambda_1| = |\Lambda_2| = n/2$ such that the random variables
$\{\delta_q^{(p,\ell)}, q \in \Lambda^i\}$ are jointly independent
for both $i=1,2$.

In the following, we will use the complex Bernstein inequality, see
for example~\cite[Proposition 15]{tr06-2} and \cite{pesh95}. It
states that for an independent sequence $\epsilon_q, q=1,\hdots,n$, of
random variables which are uniformly distributed on the torus,
\begin{equation}\label{eq:bernstein}
\P\left(\left|\sum_{q=1}^{n} \epsilon_q \right| \geq n u \right) \,\leq\,
2 e^{- n u^2/2}.
\end{equation}
Using the
pigeonhole principle and the inequality (\ref{eq:bernstein}) we obtain
\begin{eqnarray}
\P\left(|\langle \boldsymbol{M}_\ell \boldsymbol{T}_p \boldsymbol{h}^R,\boldsymbol{h}^R\rangle | \geq t \right) &=&
\P\big(\big|\sum_{q=0}^{n{-}1} \delta_{q}^{(p,\ell)} \big| \geq n t
\big) \notag \\ &\leq& \P\big(\big|\sum_{q \in \Lambda^1}
\delta_{q}^{(p,\ell)}\big| \geq nt/2 \big) +\ \P\big(\big|\sum_{q\in
\Lambda^2} \delta_q^{(p,\ell)}\big| \geq
nt /2 \big)\notag\\
& \leq& 4 \exp(-n t^2/4).\notag
\end{eqnarray}
Forming the union bound over all possible
$(p,\ell) \in \{0,\hdots,n{-}1\}^2 \setminus \{(0,0)\}$ and choosing
$t = \alpha / \sqrt{n}$
yields the statement of Theorem~\ref{lem:mu_prob} for $n$ even.

The proof of Theorem~\ref{lem:mu_prob} for $n$ odd uses essentially the same
technique as for $n$ even, with the difference that
the random variables $\delta_k^{(m,\ell)}$ are grouped
into sets of unequal cardinality,
$|\Lambda^1|= (n{-}1)/2$ and $|\Lambda^2| = (n+1)/2$.
For large $n$ the probability tail bounds are nearly the same
for $n$ even (\ref{eq:p1}) and $n$ odd (\ref{eq:p2}).\hfill $\square$


\subsection{Proof of Theorem~\ref{cor1}}

Part (a) follows directly from Theorem~\ref{thm:coh} and the coherence of $\BG \boldsymbol{h}^A$ (\ref{coh:h1}).

Part (b) follows from Theorem~\ref{thm:coh} and Theorem~\ref{lem:mu_prob}. In fact,
the probability that the condition $\mu < (2k-1)^{-1}$ of Theorem~\ref{thm:coh}
does {\em not} hold for $\BG \boldsymbol{h}^R$ is estimated by
\[
\P(\mu \geq (2k-1)^{-1}) \leq 4n^2
\exp\left(-\frac{n}{4(2k-1)^2}\right).
\]
Requiring that the latter term is less than $e^{-t}$ and
solving for $k$ gives (\ref{cond:k}).\hfill $\square$

\subsection{Proof of Theorem~\ref{thm:bp_gabor_nlogn}}\label{subsec:thm5.2}

Having established 
coherence results for  $\BG \boldsymbol{h}^A$ and $\BG \boldsymbol{h}^R$
in Section
\ref{subsec:gabor_coherence},
Theorem~\ref{thm:bp_gabor_nlogn} follows from
Theorems~\ref{thm:Tropp} and~\ref{thm:Tropp2} of Tropp \cite{tr06-2} as shown below.

(a) Recall from (\ref{coh:h1}) that the coherence for $\BG \boldsymbol{h}^A$
satisfies $\mu = n^{-1/2}$.
Next, observe that $\boldsymbol{h}^A$ unimodular implies that the columns of  $\BG \boldsymbol{h}^A$ form $n$ orthonormal bases, and, hence,  $n=\|(\BG \boldsymbol{h}^A)^\ast\|_{2,2}^2=\|\BG \boldsymbol{h}^A\|_{2,2}^2 $. Plugging this into
condition (\ref{cond:Tropp1}) of Tropp's theorem with $\delta = 1/2$ we require
that
\begin{eqnarray*}
\sqrt{144 s}\, \sqrt{\frac{k \log(k/2+1)}{n}} + \frac{2k}{n} = e^{-1/4}/2.
\end{eqnarray*}
Solving for $s$ yields (\ref{cond:s1}). Applying Theorem \ref{thm:Tropp2},
which requires $s\geq 1$, shows that condition (\ref{cond:Tropp1})
in Theorem~\ref{thm:Tropp} holds for $\boldsymbol{A}=\BG \boldsymbol{h}^A$
and we conclude that $\|\boldsymbol{A}_{\Lambda}^* \boldsymbol{A}_{\Lambda} - \Id\|_{2,2} \leq 1/2$
with probability at least $1- (k/2)^{-s}$.

Now
let $\delta \defleft \|\boldsymbol{A}_\Lambda^* \boldsymbol{A}_\Lambda - \Id\|_{2,2}$.
Then
\begin{align}
&\P(\nrecA) \notag\\
&\leq\, \P(\nrecA |\delta \leq 1/2)\notag  + \P(\delta > 1/2).\notag
\end{align}
Thus by Theorem~\ref{thm:Tropp2} we can lower bound the probability
that recovery is successful by
\[
1-((k/2)^{-s} + 2n^2 \exp(-\frac{n}{8k})).
\]
Furthermore, observe that $2n^2 \exp(-\frac{n}{8k}) \leq \epsilon$
under condition (\ref{kcond:a:1}). 

(b) Let $\mu$ be the coherence associated with the random Gabor
window $\boldsymbol{h}^R$. Setting $\alpha^2 = p \log n$ in
Theorem~\ref{lem:mu_prob} we obtain that the probability
that $\mu$ exceeds $\sqrt{\frac{p\log n}{n}}$ is smaller than
$$4n(n-1) \exp(-\alpha^2/4) \leq 4n^{-p/4+2}\,.$$ Set $\sigma = p/4-2$, i.e.,
$p=4(\sigma + 2)$, and assume for the moment that
$\mu \leq \sqrt{\frac{p \log n }{n}}$. Then condition (\ref{cond:Tropp1})
with $\delta = 1/2$ of Theorem \ref{thm:Tropp2} is satisfied if
\[
\sqrt{144 s}\sqrt{4(\sigma + 2)\frac{k\log n }{n}} + \frac{2k}{n} =
e^{-1/4}/2.
\]
Requiring $s\geq 1$ yields condition (\ref{cond:s2}). Invoking
Theorem \ref{thm:Tropp2} we obtain that
$\|\boldsymbol{A}_{\Lambda}^* \boldsymbol{A}_{\Lambda} - \Id\|_{2,2} \leq 1/2$,
$A = (\BG \boldsymbol{h}^R)$,
with probability at least $1- (k/2)^{-s}$.

Similarly to the proof of part (a), we estimate the probability
of successful recovery by
\begin{align}
& \P(\recR) \notag \\
& \geq 1 -  \Big(\ \P\big(\nrecR|\delta \leq 1/2 \ \& \ \mu^2 \leq \frac{p \log n }{n} \big) \notag \\
&\qquad \qquad\qquad \qquad+\ \P\big(\delta > 1/2 |\mu^2 \leq \frac{p\log n}{n}\big) + \ \P\big(\mu^2 > \frac{p\log n}{n}\big) \ \Big)\,.\notag
\end{align}
By Theorem~\ref{thm:Tropp},
the probability that $\op$ can be reconstructed from
$\op \boldsymbol{h}^R$ by Basis Pursuit (\ref{eq:l1}) exceeds
\[
1- (2n^2 \exp(-\frac{n}{8p\log(n) k}) + (k/2)^{-s} + 4n^{-\sigma}).
\]
Finally, observe that the term $2n^2\exp(-\frac{n}{p\log(n) k})$
is less than $\epsilon$ provided
\[
k \leq \frac{n}{32(\sigma+2)\log(n)\log(2n^2/\epsilon)}.
\]

\subsection{Proof of Corollary~\ref{cor:Fourier}}\label{subsec:cor2.6}

Plancherel's theorem and $\widehat{\boldsymbol{M}_\ell \boldsymbol{T}_p \boldsymbol{h}}=\boldsymbol{T}_\ell \boldsymbol{M}_{n-p}\widehat{\boldsymbol{h}}= \sigma \boldsymbol{M}_{n-p}\boldsymbol{T}_\ell \widehat{\boldsymbol{h}}$ with $|\sigma| = 1$ implies that
the coherence remains the same under Fourier transform
of the window, that is,
\begin{align}
\mu_h & = \sup_{(\ell,p) \neq (\ell',p')}|\langle \boldsymbol{M}_\ell \boldsymbol{T}_p \boldsymbol{h}, \boldsymbol{M}_{\ell'} \boldsymbol{T}_{p'} h\rangle| = \sup_{(\ell,p) \neq (\ell',p')} |\langle \widehat{\boldsymbol{M}_\ell \boldsymbol{T}_p \boldsymbol{h}}, \widehat{\boldsymbol{M}_{\ell'} \boldsymbol{T}_{p'} \boldsymbol{h}}\rangle|\notag\\
&=\, \sup_{(\ell,p) \neq (\ell',p')} |\langle \boldsymbol{M}_{n-p} \boldsymbol{T}_{\ell} \hat{\boldsymbol{h}}, \boldsymbol{M}_{n-p'} \boldsymbol{T}_{\ell'} \hat{\boldsymbol{h}} \rangle| \,=\, \mu_{\hat{h}}.\notag
\end{align}
Since all of the results concerning the dictionary of
time-frequency shift matrices stated above are based on the coherence this proves the
claim.

\section{Multiple test vectors}
\label{sec:severalvectors}

In addition to the goal of recovering the operator
$\op$ from the operator output caused by a single test signal, we may also
consider using two or more test signals $\boldsymbol{h}_1,\hdots,\boldsymbol{h}_r$ to identify $\op$.
In this case, the vector
of concatenated observations $\op \boldsymbol{h}_1,\hdots, \op \boldsymbol{h}_r$ is given as
\[
\left( \begin{matrix} \op \boldsymbol{h}_1 \\ \vdots \\ \op \boldsymbol{h}_r \end{matrix}
\right)
\,=\, \left( \begin{matrix} \BPsi_1 \boldsymbol{h}_1 & \hdots & \BPsi_N \boldsymbol{h}_1 \\
\vdots & & \vdots \\
\BPsi_1 \boldsymbol{h}_r & \hdots & \BPsi_N \boldsymbol{h}_r \end{matrix} \right) \boldsymbol{x} =
\left( \begin{matrix} \BPsi \boldsymbol{h}_1 \\ \vdots \\ \BPsi \boldsymbol{h}_r \end{matrix}\right)\boldsymbol{x},
\]
and our sparse matrix identification task is again reduced to a sparse signal
recovery problem. Although we will not pursue this task in depth here,
we will make some remarks and state extensions of our results to this
more general setting.

Intuitively, using several test vectors instead of a single one should
increase the maximal sparsity $k$ that allows for perfect reconstruction
as more information can be exploited.
However, it is only interesting to consider $r < m$ since
any operator $\boldsymbol{\Gamma} \in \C^{n \times m}$
can be characterized by its action on $m$ basis vectors.
The following lemma on coherence of concatenated measurement matrices
suggests that the maximal recoverable sparsity does not
decrease. Its proof is straightforward and therefore omitted.

\begin{lemma} Let $\boldsymbol{h}_1,\hdots, \boldsymbol{h}_r \in \C^m$
such that the matrices $(\BPsi \boldsymbol{h}_j)$ have
coherence $\mu_j$. Then the coherence $\mu$ of the
normalized concatenated matrix
\begin{align}
\boldsymbol{A}_{\boldsymbol{h}_1,\hdots,\boldsymbol{h}_r}
& = \frac{1}{\sqrt{r}}\left(\begin{array}{c} (\BPsi \boldsymbol{h_1}) \\
(\BPsi \boldsymbol{h}_2) \\  \vdots \\
(\BPsi \boldsymbol{h}_r) \end{array} \right) =\, \frac{1}{\sqrt{r}}
\left( \begin{matrix} \BPsi_1 \boldsymbol{h}_1 & \hdots & \BPsi_N \boldsymbol{h}_1 \\
\vdots & & \vdots \\
\BPsi_1 \boldsymbol{h}_r & \hdots & \BPsi_N \boldsymbol{h}_r \end{matrix} \right)\notag
\end{align}
satisfies $\mu \leq \frac{1}{r}(\mu_1 + \mu_2 + \cdots + \mu_r) \leq \max_{j=1,\hdots,r} \mu_j$.
\end{lemma}

A straightforward
extension of the proof of Theorem \ref{lem:mu_prob} yields the following result in the setting of time-frequency shifts and several randomly chosen
$\boldsymbol{h}^R_j$, $j=1,\hdots,r$.

\begin{theorem} Let $n \in \N$ be even and
choose random windows $\boldsymbol{h}^R_j$, $j=1,\hdots,r$,
with entries
\begin{equation*}
(\boldsymbol{h}^R_j)_q \,\defleft\, \frac{1}{\sqrt{n}} \epsilon_{qj}, \quad q=0,\hdots,n{-}1,
\end{equation*}
where the $\epsilon_{qj}$ are independent and uniformly distributed on
the torus $\{z \in \CC, |z|=1\}$. Let $\mu$ be the coherence of the
concatenated matrix
\[
\frac{1}{\sqrt{r}}\left(\begin{matrix} (\BG \boldsymbol{h}^R_1) \\ \vdots \\ (\BG \boldsymbol{h}^R_r) \end{matrix}\right)
\]
where $\BG$ is defined in (\ref{eq:bg}).
Then for $\alpha > 0$
\begin{equation}\label{eq:p1}
\P\big( \mu \geq \frac{\alpha}{\sqrt{rn}}\big) \leq 4 n(n{-}1) e^{-\alpha^2/4}.
\end{equation}
\end{theorem}
Similarly as in Theorem~\ref{cor1}(b) we deduce that the condition
\[
k \leq \frac{1}{4} \sqrt{\frac{rn}{2 \log n + \log 4 + t}}
\]
implies that Basis Pursuit (or Orthogonal Matching Pursuit) recovers
all $k$-sparse
$\boldsymbol{\Gamma}$ from $\boldsymbol{\Gamma} \boldsymbol{h}_1^R, \hdots, \boldsymbol{\Gamma} \boldsymbol{h}_r^R$ with probability at least
$1-e^{-t}$.
Hence, the maximal provable sparsity increases at least by a factor of $\sqrt{r}$.

Of course, we may as well apply Tropp's result based on random support sets
and phases to arrive at a statement analogous to Theorem
\ref{thm:bp_gabor_nlogn}.

\begin{theorem}\label{thm:bp_gabor_nlognMULT}
Let $n$ be even and $k \geq 3$ and let $\Lambda$ be chosen uniformly at random
among all subsets of $\{0,\hdots,n{-}1 \}^2$ of cardinality $k$. Suppose
further that $\boldsymbol{x} \in \C^n$ has support $\Lambda$ with random phases
$(\sgn(x_{\ell p}))_{(\ell,p)\in \Lambda}$ that are independent and
uniformly distributed on the torus $\{z, |z|=1\}$. Let
\[ \op =
\sum_{(\ell,p) \in \Lambda} x_{\ell p} \boldsymbol{M}_\ell \boldsymbol{T}_p.
\]
Choose $r$ independent random
windows $\boldsymbol{h}^R_1,\hdots,\boldsymbol{h}^R_r$ according to
(\ref{eq:h2}).
Assume
\begin{equation*}
k \leq \frac{rn}{32(\sigma + 2) \log n \log(2n^2/\epsilon)}
\end{equation*}
for some $\sigma > 0$ and
\begin{eqnarray}
s &:= & \frac{1}{576(\sigma + 2)} \left(e^{-1/4}/2 - \frac{2k}{n}\right)^2  \cdot \ \frac{rn}{k \log(k/2+1)} \geq 1\label{cond:s2}\,.
\end{eqnarray}
Then with probability at least
\[
1 - (\epsilon + 4n^{-\sigma}  + (k/2)^{-s})
\]
Basis Pursuit (\ref{eq:l1}) recovers $\op$ from $\op \boldsymbol{h}^R_1,\hdots,\op \boldsymbol{h}^R_r$.
\end{theorem}

Roughly speaking, with the chosen probabilistic model on the
sparse coefficient vector $\boldsymbol{x}$, the provable maximal sparsity $k$ that allows for recovery, increases by a factor of $r$ when taking $r$ test vectors instead of only one. This fact is illustrated in Figure~\ref{fig:logistic2} in Section~\ref{sec:numerics}.

\section{Numerical results}\label{sec:numerics}

Theorem~\ref{thm:bp_gabor_nlogn} can be tested empirically for
various values of $n$ by trying a number of sparsity levels $k$ and
recording the fraction of times (\ref{eq:l1}) recovers
the true $k$-sparse coefficient vector $\boldsymbol{x}$.

But before doing so, we illustrate in Figure~\ref{fig:example} the recovery method for matrices which have a sparse representation in the dictionary of time--frequency shift matrices as considered in Theorem~\ref{thm:bp_gabor_nlogn}. A $7$-sparse coefficient vector $\boldsymbol{x}$ in the time-frequency plane is chosen and reconstructed
from $\boldsymbol{\Gamma} \boldsymbol{h}^A = \sum_{\ell,p} x_{\ell p} \boldsymbol{M}_\ell \boldsymbol{T}_p \boldsymbol{h}^A$ by Basis Pursuit. As comparison, $\boldsymbol{x}$ is reconstructed by a traditional reconstruction by $\ell_2$-minimization,
\begin{equation}
 \min \|\boldsymbol{x}\|_2 \text{ subject to } (\boldsymbol{\Psi} \boldsymbol{h}^A) x = \boldsymbol{\Gamma} \boldsymbol{h}^A\,.\label{eq:l2minimization}
\end{equation}

\begin{figure}
\begin{center}
\begin{minipage}{6cm}\begin{center}
\includegraphics[width=2.5in]{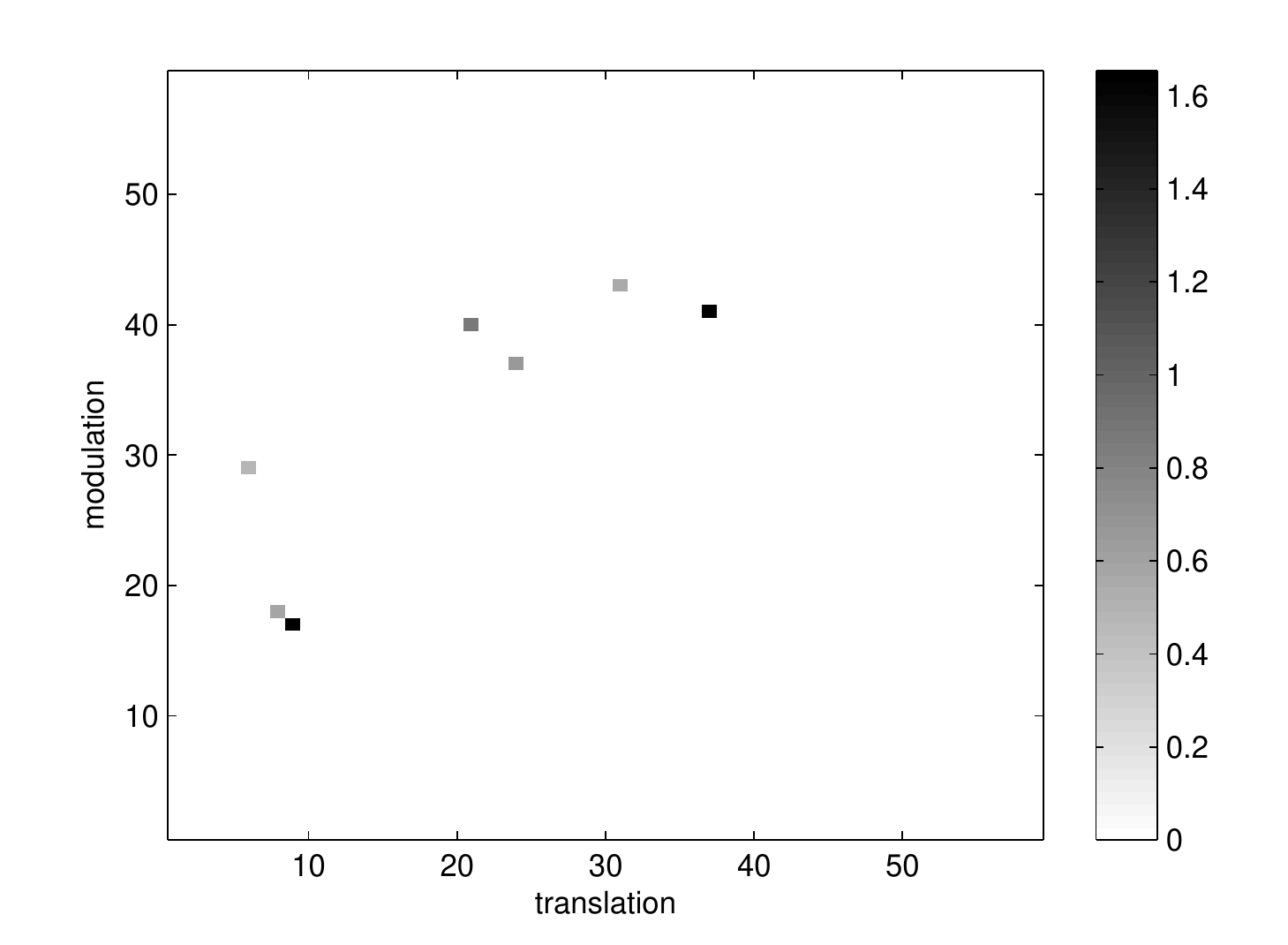}

(a)\end{center}
\end{minipage}\begin{minipage}{6cm}\begin{center}

\includegraphics[width=2.5in]{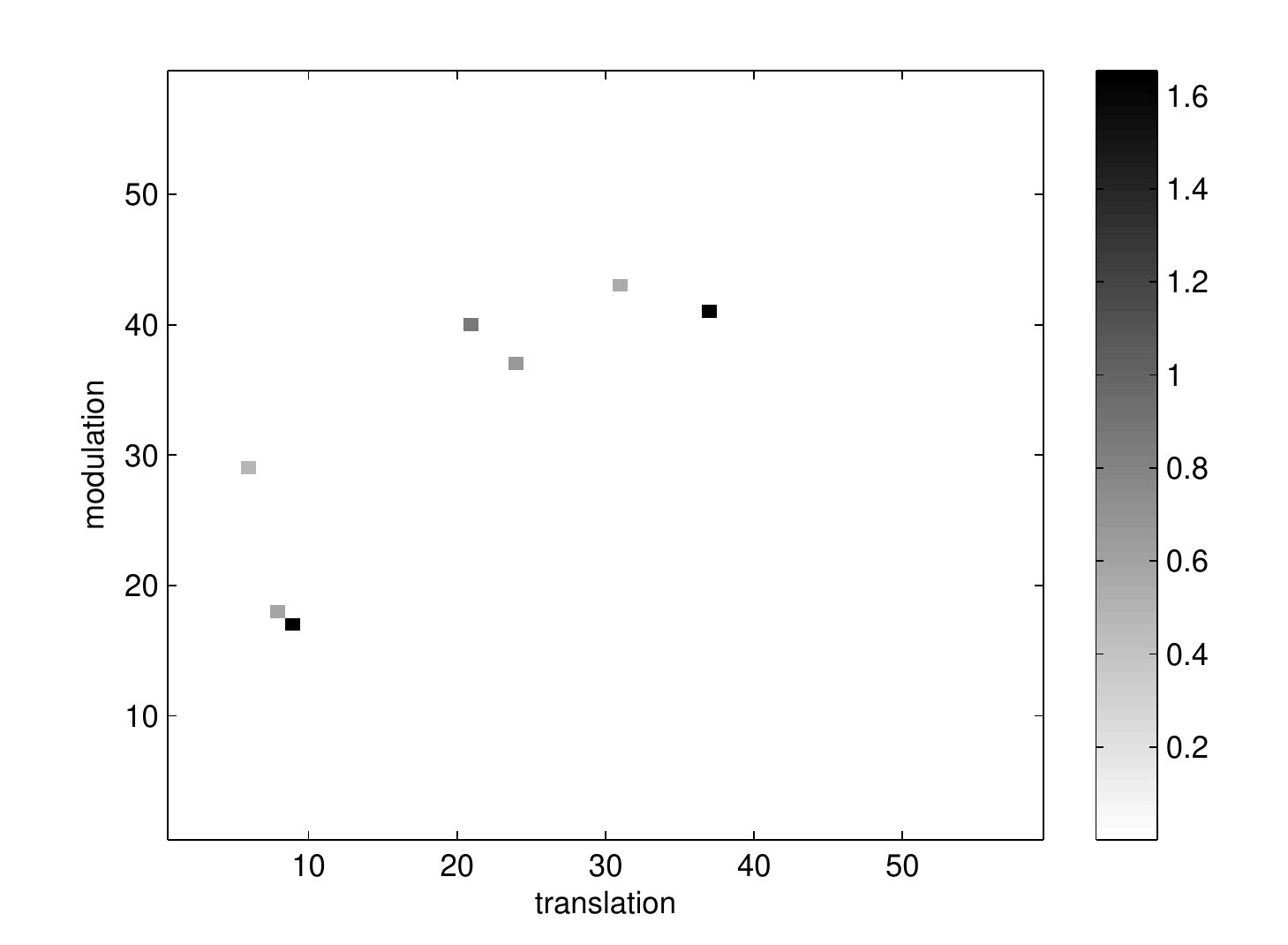}

(b)
\end{center}
\end{minipage}\begin{minipage}{6cm}\begin{center}
\includegraphics[width=2.5in]{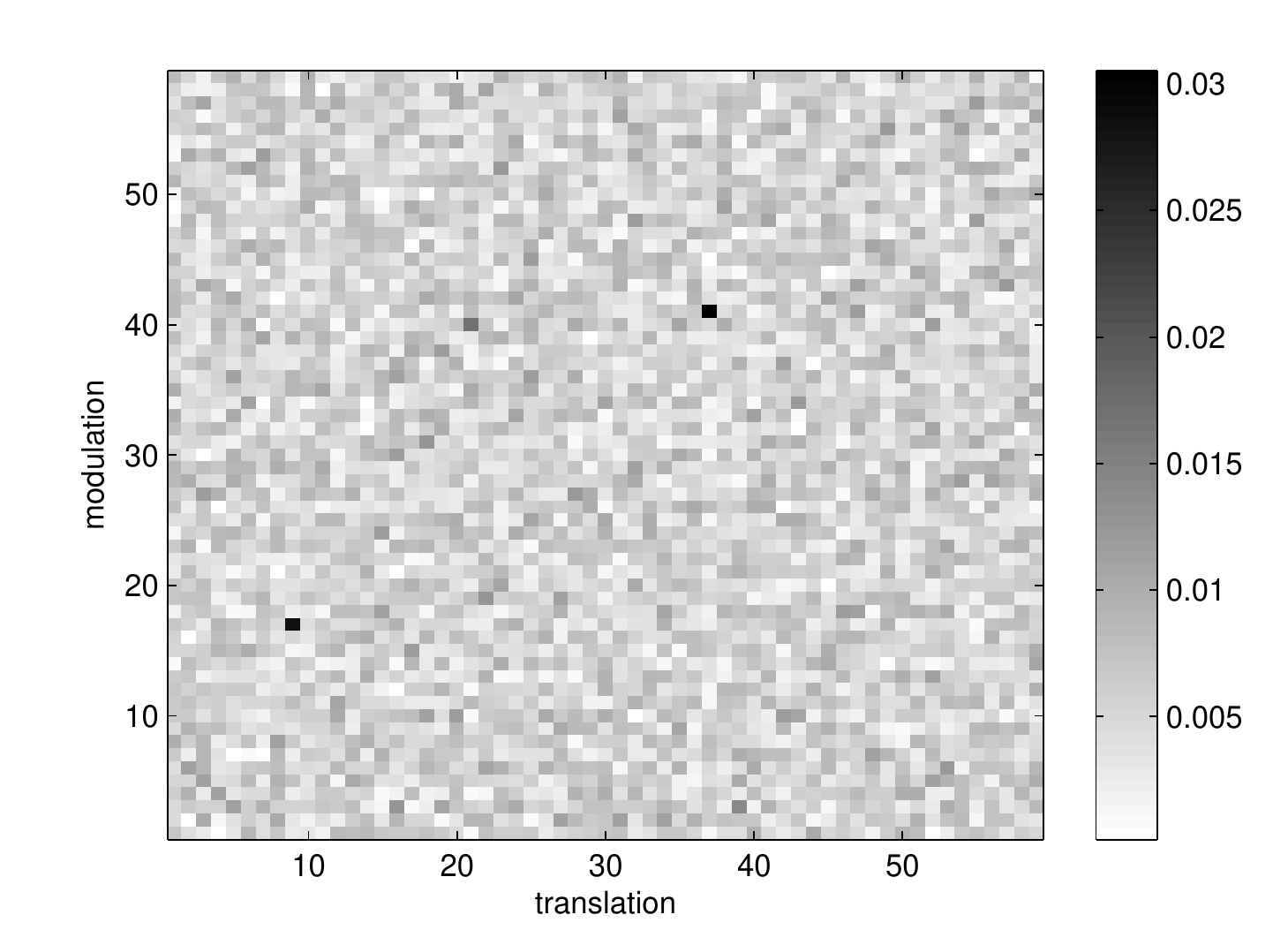}

(c)\end{center}
\end{minipage}
\end{center}
\caption{(a) Original $7$-sparse coefficient vector ($n=59$) in the time-frequency plane.
(b) Reconstruction by Basis Pursuit using the Alltop window $\boldsymbol{h}^A$.
(c) For comparison, the reconstruction by traditional $\ell_2$-minimization \eqref{eq:l2minimization}.}
\label{fig:example}
\end{figure}

For the Alltop window $\boldsymbol{h}^A$ in (\ref{eq:h1}) we consider the
values of $n$ prime from 11 to 59, for the random window $\boldsymbol{h}^R$ in
equation (\ref{eq:h2}) we consider the values of $n$ prime from 11
to 59 as well as $n=10+4j$ for $j=0,1,\ldots, 12$. Each empirical
test consists of generating a random $k$-sparse $x\in\CC^{n^2}$ with
non-zero entries $x_q=r_q\exp(2\pi i \theta_q)$,
with $r_q$ drawn independently from the Gaussian $N(0,1)$ distribution, and
$\theta_q$ drawn independently and uniformly from $[0,1)$.

For each value of $n$, 1000 tests are computed per value of $k=1,2,\ldots,n{-}1$. A test is considered successful if Basis Pursuit (\ref{eq:l1}) recovers all components of the coefficient vector $\boldsymbol{x}$ with $10^{-10}$ error tolerance.  The successful recovery of $\boldsymbol{x}$, and, hence, of $\op$ from $\op \boldsymbol{h}^A$ or $\op \boldsymbol{h}^R$ is recorded in $Y_k^n$ as a 1, and failure to recover as a 0.
 Following the empirical examination of phase transitions in
\cite{dots06-1}, we approximate the observed probability distribution
by fitting the mean response of $Y_k^n$ using the logistic
regression model, \cite{hatifr01},
\begin{equation}\label{eq:logistic}
E(Y_k^{n}) = \frac{\exp (\beta_0(n) + \beta_1(n) k)}{1+\exp (\beta_0(n) + \beta_1(n) k)}.
\end{equation}

For illustration purposes, the fitted response for windows $\boldsymbol{h^A}$
with $n=43$ and $\boldsymbol{h^R}$ with $n=30$ is shown in
Figure~\ref{fig:logistic} along with the mean response of $Y_k^n$.

\begin{figure}
\begin{center}
\begin{tabular}{c}
\includegraphics[width=2.5in]{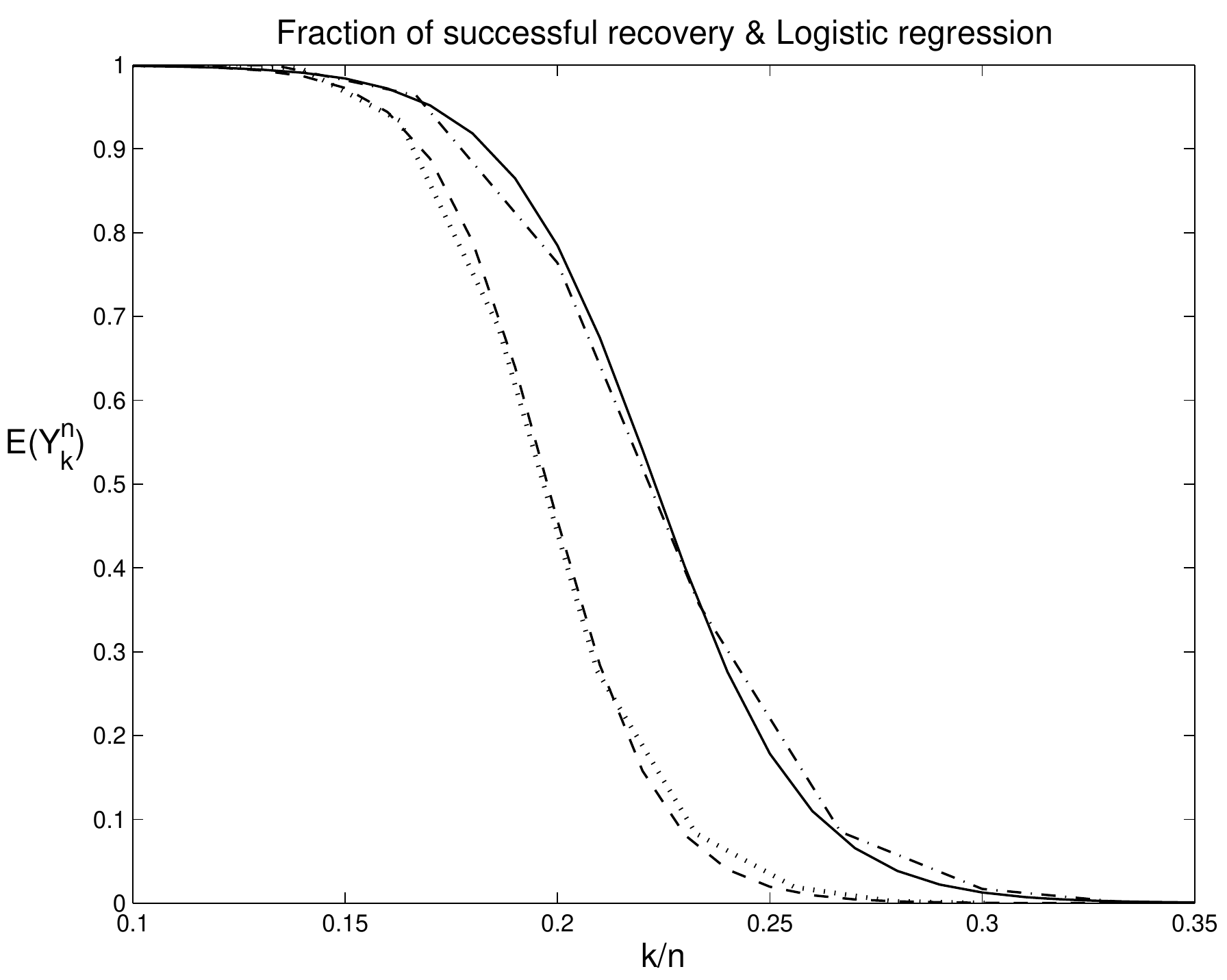}
\end{tabular}
\end{center}
\caption{Empirical verification of Theorem~\ref{thm:bp_gabor_nlogn}
 without noise.
For the random window $\boldsymbol{h}^R$ with $n=30$
the mean response of $Y_k^n$ (dash-dot) and fitted logistic
regression model $E(Y^n_k)$, (solid), plotted against the
fractional sparsity $k/n$.
For the Alltop window $\boldsymbol{h}^A$ with $n=43$
the mean response of $Y_k^n$ (dot) and fitted logistic
regression model $E(Y^n_k)$, (dash), plotted against the
fractional sparsity $k/n$.}\label{fig:logistic}
\end{figure}

The phase transition behaviors are often observed through
the fractional sparsity ratio $k/n$, and the matrix so-called undersampling
rate $n/N$, here $1/n$ for $\BG \boldsymbol{h}^A$ and $\BG \boldsymbol{h}^R$ \cite{dota06}.  Contours of the fitted logistic
regression models for time-frequency shift dictionaries with identifiers
$\boldsymbol{h}^A$ and $\boldsymbol{h}^R$ are shown in Figure
~\ref{fig:contour} (a) and (b) respectively.  To facilitate a
quantitative inspection of the contours in Figure~\ref{fig:contour} and
the theoretical results of \cite{dota06} we overlay the
contours in Figure~\ref{fig:contour} with the level curve for $93\%$ success
rate (dash) and $1/(2\log n)$ (solid).  The curve $1/(2\log n)$ is known to be
the threshold for overwhelming
probability of successful recovery in the case of Gaussian random
matrices for large $n$ \cite{dota06}.
It is observed in Figure~\ref{fig:contour}
that the curve $1/(2\log n)$ remains below the $93\%$
success rate level curve, indicating consistence of the empirical
results with the phase transition $1/(2\log n)$ conjectured for the
class of time-frequency shift matrices applied to identifiers
$\boldsymbol{h}^A$ and $\boldsymbol{h}^R$.
Moreover, the curve $1/(2\log n)$ increasingly falls below the
$93\%$ success rate level curve as $n$ increases, indicating improved
agreement in the large $n$ limit.
Note that this conjectured
phase transition $1/(2\log n)$ is larger than that proven in the main Theorem
~\ref{thm:bp_gabor_nlogn}, both in order (as $u=0$ here), as well as in
the constant.

\begin{figure}
\begin{center}\begin{minipage}{6cm} \begin{center}
  \includegraphics[width=2.5in]{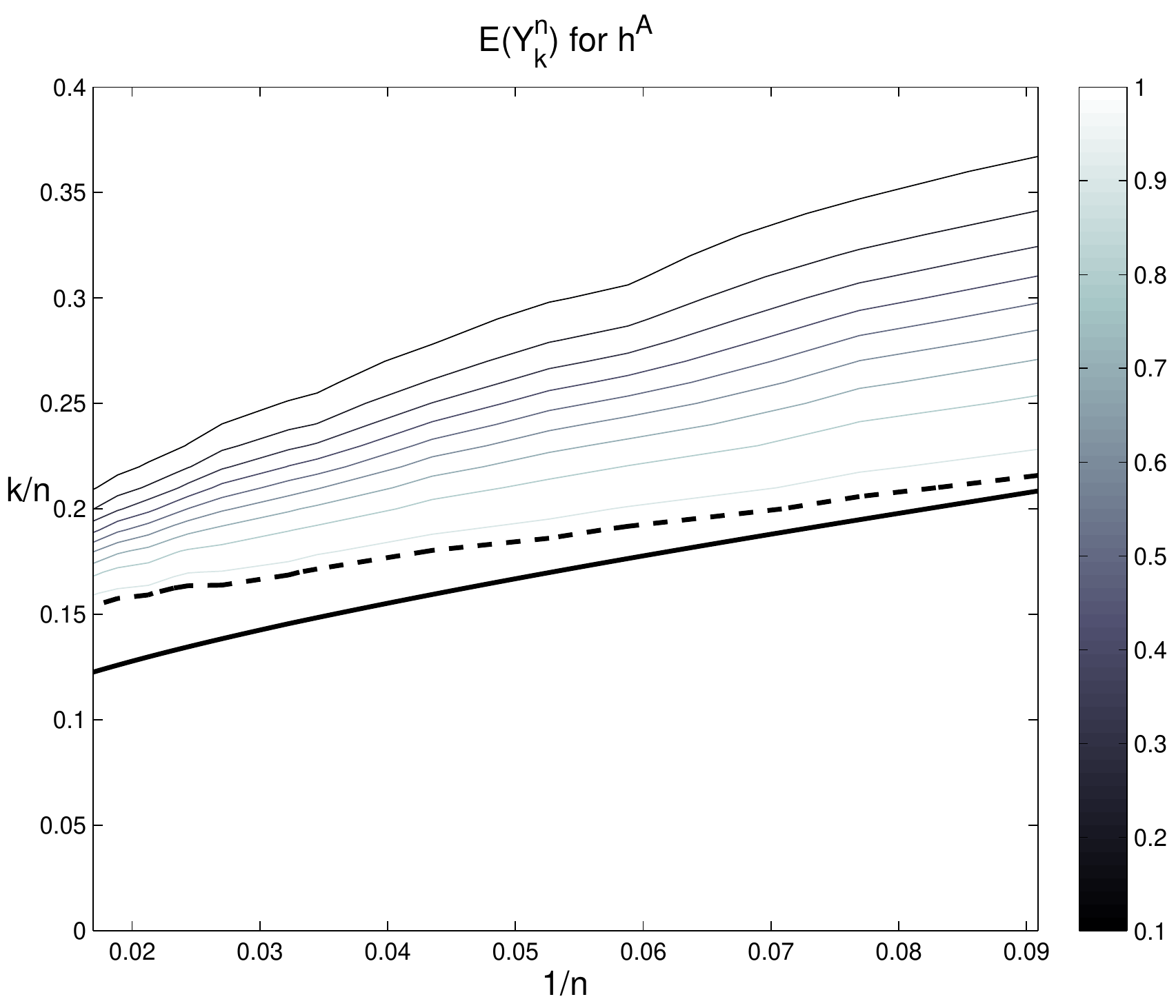}

(a)
\end{center}
\end{minipage}\qquad
\begin{minipage}{6cm} \begin{center}
\includegraphics[width=2.5in]{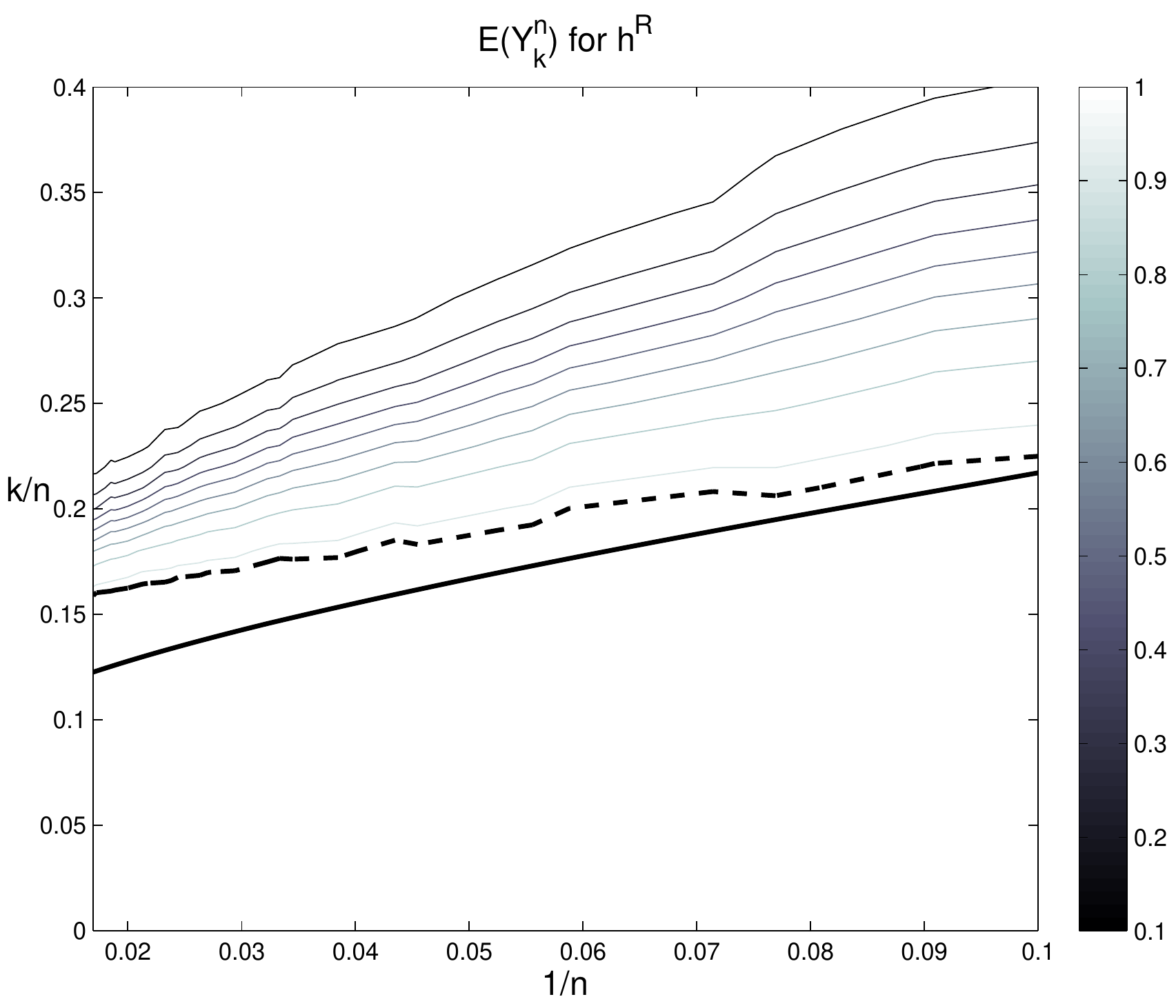}

(b)
\end{center}
\end{minipage}

\end{center}
\caption{Empirical verification of Theorem~\ref{thm:bp_gabor_nlogn}
for $\boldsymbol{h}^A$ (a) and $\boldsymbol{h}^R$ (b) without noise.
Contours of the fitted logistic regression model (gray), the $93\%$
success rate contour (dashed), and $1/(2\log n)$ (solid).
Figure~\ref{fig:logistic} shows vertical slices
for $1/43$ (a) and $1/30$ (b).}
\label{fig:contour}
\end{figure}

As stated earlier, in practice the measurements $\op \boldsymbol{h}$ are observed
with noise and although $\op$ can be well approximated by a $k$-sparse
representation, it is rarely strictly $k$-sparse.  For both of these
reasons, the recovery algorithm (\ref{eq:l1}) is not often used in practice,
rather (\ref{eq:l1noise}) is used to allow for an inexact fit of the
measurements.

In Figure~\ref{fig:contour_noise} we empirically test Theorem~\ref{thm:bp_gabor_nlogn} using
(\ref{eq:l1noise}) rather than (\ref{eq:l1}) for the reconstruction
algorithm. We choose the same values of $k$ and $n$, and the same number of tests were performed as
for Figure~\ref{fig:contour}. The non-zero entries in $\boldsymbol{x}$ are also selected from the same
distribution as was used to generate Figure~\ref{fig:contour}. Additive noise is simulated at a level of 25 dB signal to noise ratio; that is, $\boldsymbol{\eta}$
is added to $\op \boldsymbol{h}$ with the
entries in $\boldsymbol{\eta}$ drawn independently from the Gaussian $N(0,1)$
and $\boldsymbol{\eta}$ is normalized to
$\|\boldsymbol{\eta}\|_2=\|\op \boldsymbol{h}\|_2\cdot 10^{-5/4}$.

\begin{figure}
\begin{center}\begin{minipage}{6cm} \begin{center}
  \includegraphics[width=2.5in]{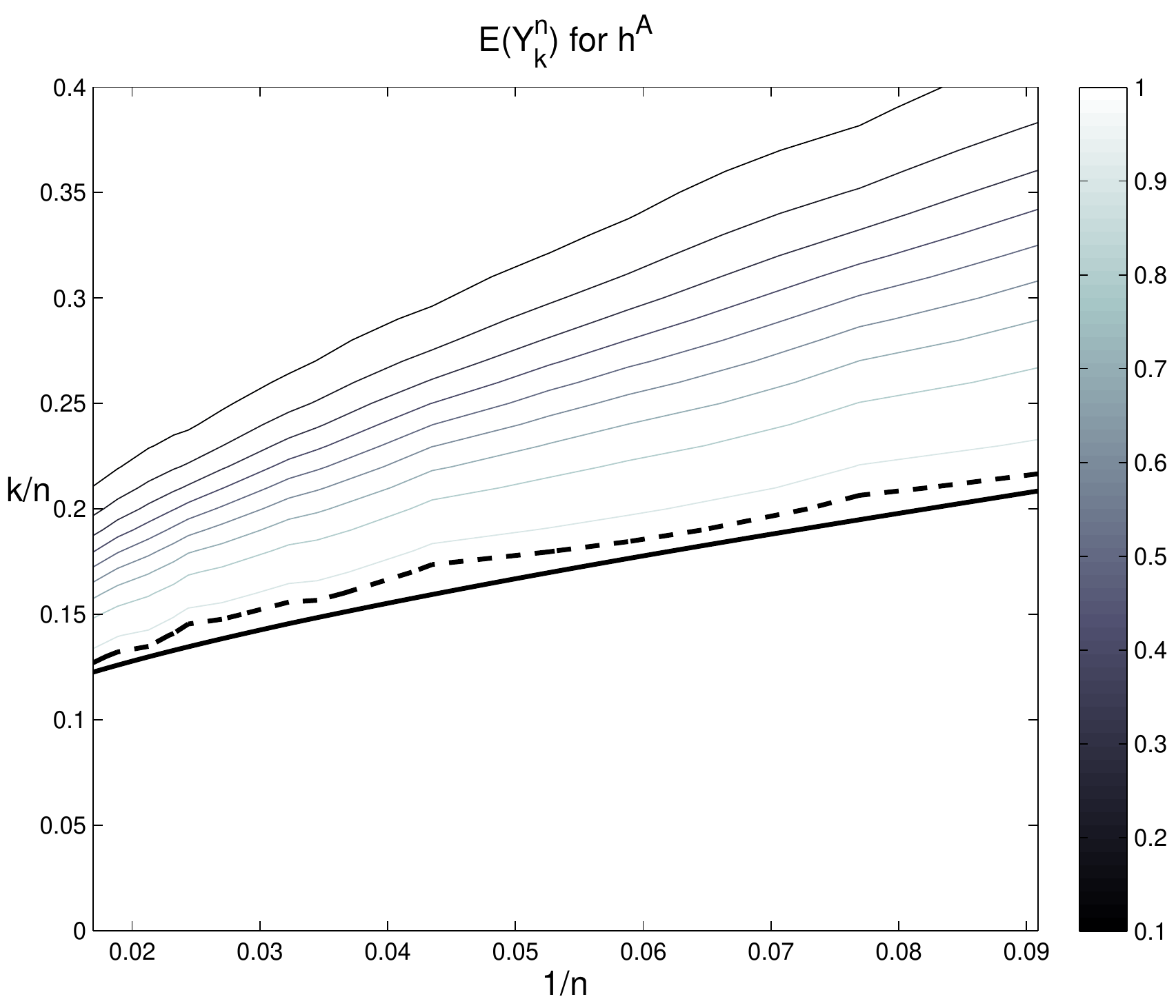}

(a)
\end{center}
\end{minipage}
\qquad
\begin{minipage}{6cm} \begin{center}
\includegraphics[width=2.5in]{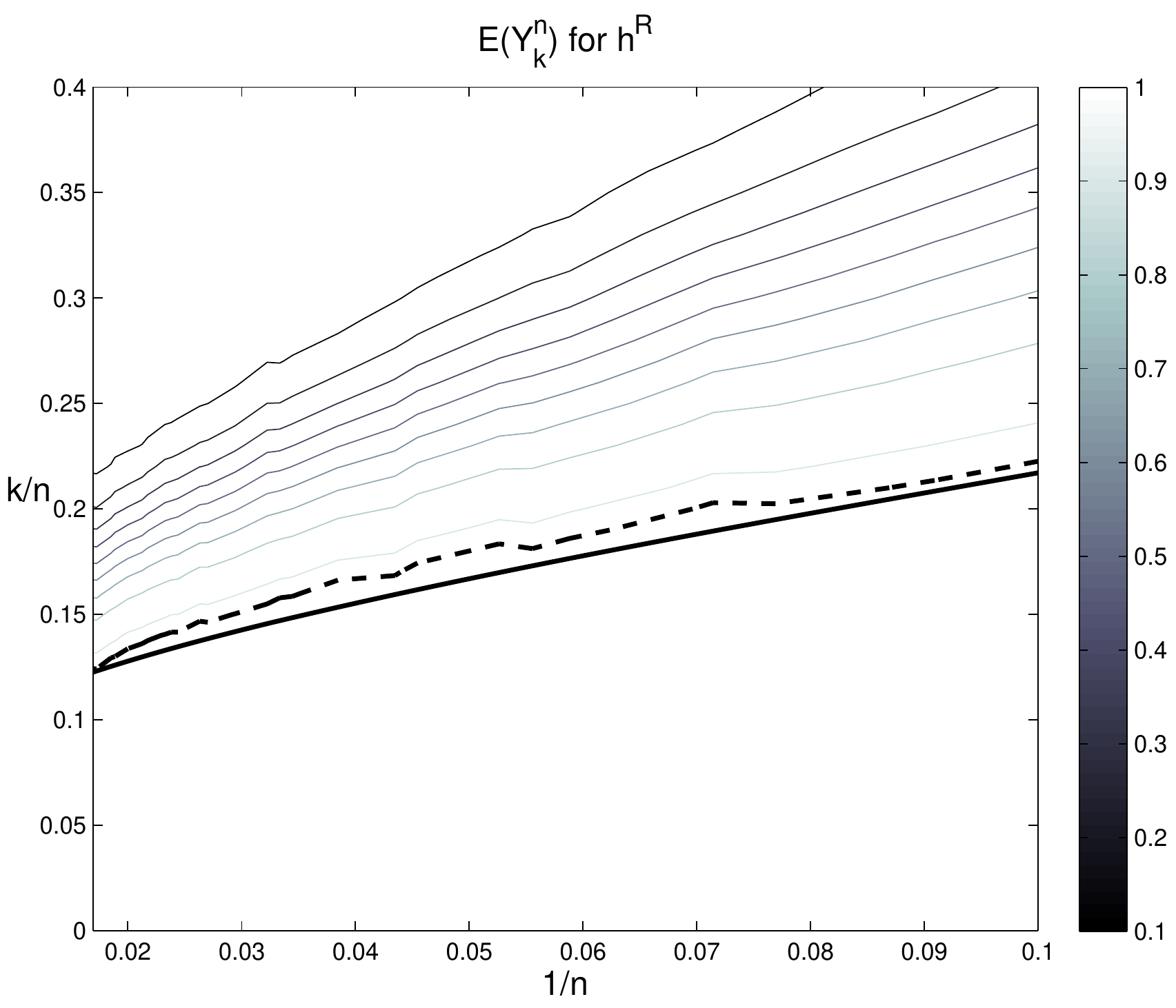}

(b)
\end{center}
\end{minipage}

\end{center}
\caption{Empirical verification of Theorem~\ref{thm:bp_gabor_nlogn}
for $\boldsymbol{h}^A$ (a) and $\boldsymbol{h}^R$ (b) in the noisy setting, with (\ref{eq:l1}) replaced by (\ref{eq:l1noise})
and additive noise of $25$ dB signal to noise ratio.
Contours of the fitted logistic regression model (gray), the $93\%$
success rate contour (dash), and $1/(2\log n)$ (solid).}
\label{fig:contour_noise}
\end{figure}

Unlike the solution of (\ref{eq:l1}) for which the exact solution can be
exactly $k$-sparse, and for which numerical algorithms can compute
approximations of arbitrary precision, the solution of (\ref{eq:l1noise})
from noisy measurements will not recover the solution exactly.  For our
numerical experiments involving noisy measurements,
the vector $\boldsymbol{x}$ associated with $\op$ resulting from the solution
of (\ref{eq:l1noise}) is only considered to have been
successfully recovered if the largest $k$ entries of the recovered $\boldsymbol{x}'$
have the same support set $\Lambda$ as $\boldsymbol{x}$.  Alternative metrics of
successful recovery, such as $\ell^2$ error or Signal to Noise Ratio (SNR), are
less demanding than requiring a match of the
support set; moreover, the support set metric was previously examined in this
setting by Wainwright \cite{wa06} and following this convention allows for a
more direct comparison.  The inequality fit parameter
$\epsilon$ in (\ref{eq:l1noise}) is selected to be at the noise
level $10^{-5/4}$.

As in the noiseless setting, we approximate the
probability distribution of the empirical observations $Y_k^n$ using the
logistic regression model (\ref{eq:logistic}).
Contours of the fitted logistic
regression models for time-frequency shift dictionaries with identifiers
$\boldsymbol{h}^A$ and $\boldsymbol{h}^R$ are shown in Figure~\ref{fig:contour_noise} (a) and (b) respectively.  Overlaying these
contours is the level curve for $93\%$ success rate (dash) and
$1/(2\log n)$ (solid).  Unlike the noiseless case (\ref{eq:l1}), it was shown that
the threshold for overwhelming probability of successful recovery in the
case of Gaussian random $n \times n^2$ matrices with noise using (\ref{eq:l1noise}) is
$1/(4\log n)$, \cite{wa06}; however, we observe in Figure~\ref{fig:contour_noise}
that $1/(2\log n)$ fits the empirical data better in this instance.  As Wainwright
considered the Gaussian setting, this empirical observation for the Gabor system does not
contradict results in \cite{wa06}, but the difference is noteworthy.

In Figure~\ref{fig:logistic2} we illustrate the performance of Basis Pursuit when using multiple test signals as discussed in Section~\ref{sec:severalvectors}, in particular in Theorem~\ref{thm:bp_gabor_nlognMULT}. Figure~\ref{fig:logistic2} was obtained using the same procedure that provided Figure~\ref{fig:logistic}.

\begin{figure}
\begin{center}
\begin{tabular}{c}
\includegraphics[width=2.5in]{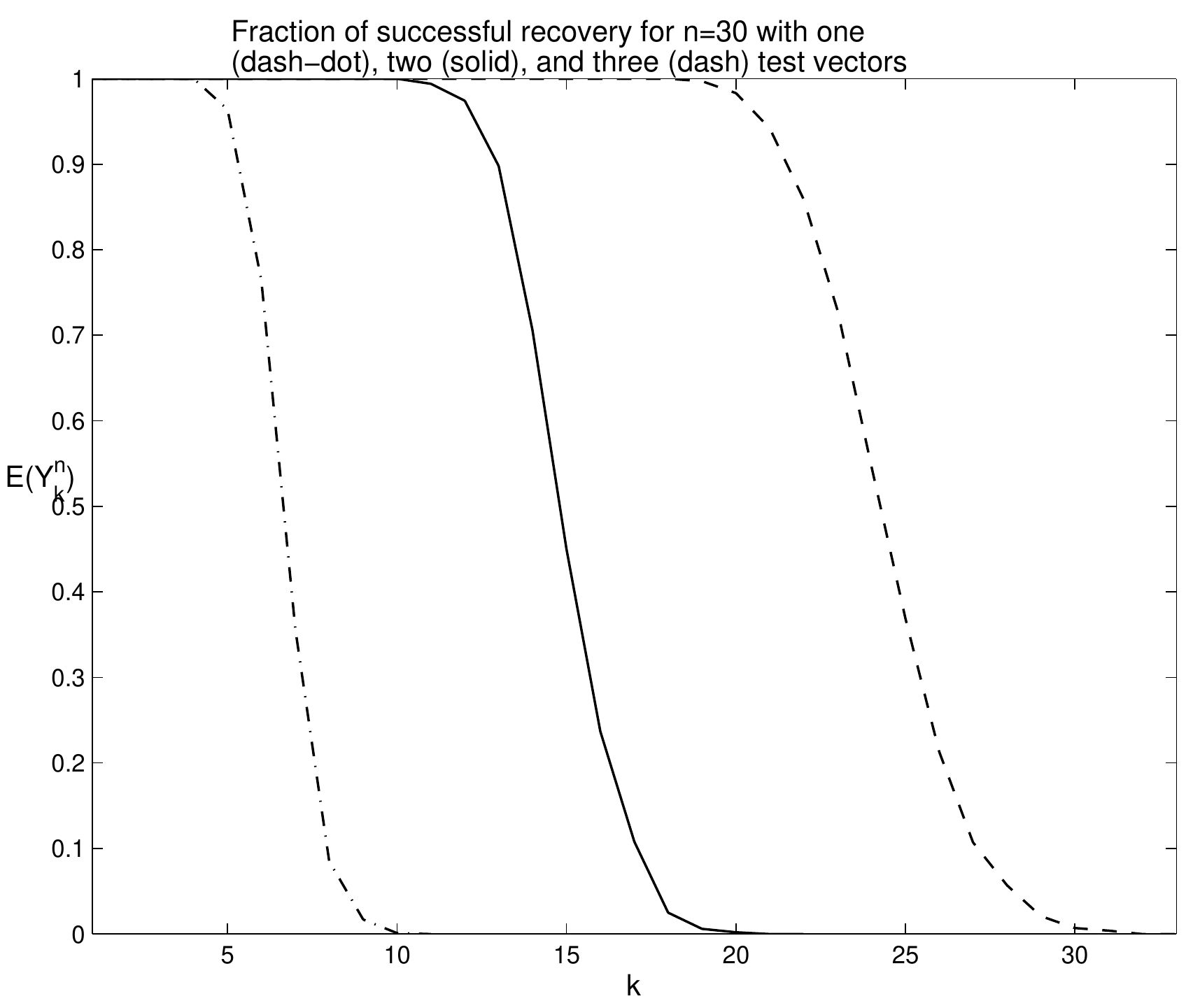}
\end{tabular}
\end{center}
\caption{Empirical verification of Theorem~\ref{thm:bp_gabor_nlognMULT}
 without noise.
For the random windows $\boldsymbol{h}_1^R, \boldsymbol{h}_2^R, \boldsymbol{h}_3^R$ with $n=30$
the fraction of successful recovery  based on $\BG \boldsymbol{h}_1^R$
(dash-dot), $\BG \boldsymbol{h}_1^R$ and $\BG \boldsymbol{h}_2^R$ (solid), and $\BG \boldsymbol{h}_1^R$, $\BG \boldsymbol{h}_2^R$ and $\BG \boldsymbol{h}_3^R$ (dash) test vectors.}\label{fig:logistic2}
\end{figure}

\bibliography{SparseBib}
\bibliographystyle{abbrv}


\end{document}